\documentclass[12pt]{article}
\usepackage{amsmath,amssymb,amsthm,xypic,stmaryrd,graphicx}
\newcommand{\re}{\mathbb{R}}
\newcommand{\co}{\mathbb{C}}
\newcommand{\fd}{\mathbb{K}}

\newcommand{\na}{\mathbb{N}}

\newcommand{\bbb}{\mathcal{B}}
\newcommand{\aaa}{\mathcal{A}}
\newcommand{\cc}{\mathcal{C}}
\newcommand{\ld}{\mathbb{F}}

\newcommand{\LL}{\mathcal{L}}
\newcommand{\SSS}{\mathcal{S}}

\newcommand{\pp}{\mathcal{P}}

\newcommand{\ttt}{\mathcal{T}}
\newcommand{\uuu}{\mathcal{U}}

\newcommand{\cpn}{{\co}P^n}
\newcommand{\rpn}{{\re}P^n}

\newcommand{\ba}{{\mbox{\boldmath$\alpha$}}}

\newcommand{\bbeta}{{\mbox{\boldmath$\beta$}}}

\newcommand{\lcm}{\mbox{\rm lcm}}
\newcommand{\boxx}{\rule{2.12mm}{3.43mm}}

\long\def\symbolfootnote[#1]#2{\begingroup%
\def\thefootnote{\fnsymbol{footnote}}\footnote[#1]{#2}\endgroup}

\newtheorem{thm}{Theorem}[section]
\newtheorem{prop}[thm]{Proposition}
\newtheorem{lem}[thm]{Lemma}
\newtheorem{cor}[thm]{Corollary}

\theoremstyle{definition}

\newtheorem{defn}[thm]{Definition}
\newtheorem{notation}[thm]{Notation}
\newtheorem{example}[thm]{Example}

\theoremstyle{remark}

\begin{document}

\title{Weighted Projective Spaces and a Generalization of Eves' Theorem}

\author{Adam Coffman}

\LaTeXdiagrams

\maketitle

\begin{abstract}
  For a certain class of configurations of points in space, Eves'
  Theorem gives a ratio of products of distances that is invariant
  under projective transformations, generalizing the cross-ratio for
  four points on a line.  We give a generalization of Eves' theorem,
  which applies to a larger class of configurations and gives an
  invariant with values in a weighted projective space.  We also show
  how the complex version of the invariant can be determined from
  classically known ratios of products of determinants, while the real
  version of the invariant can distinguish between configurations that
  the classical invariants cannot.

\end{abstract}

\section{Introduction}\label{sec0}\symbolfootnote[0]{MSC 2010: Primary
    51N15; Secondary 05B30, 14E05, 14N05, 51A20, 51N35, 68T45}

Eves' Theorem (\cite{e}) is a generalization of two basic geometric
results: Ceva's Theorem for triangles in Euclidean geometry, and the
projective invariance of the cross-ratio in projective geometry.  Both
results, and more generally Eves' Theorem, assign an invariant ratio
of products of distances to certain types of configurations of points
in space.  

\begin{center}
  \begin{figure}
\begin{center}
    \includegraphics[scale=0.65]{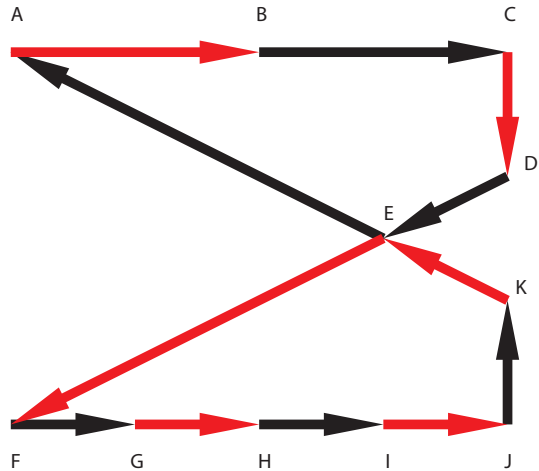}
    \caption{A configuration of $11$ points, $5$ lines, and $12$
      directed segments in the real Euclidean plane, to which Eves'
      Theorem applies.}\label{fig1}
  \end{center}
  \end{figure}
\end{center}

The example shown in Figure \ref{fig1}, where eleven points lie on
five lines, forming twelve directed segments, gives the general idea
of Eves' Theorem.  The ratio of Euclidean signed
distances $$\frac{AB\cdot CD\cdot EF\cdot GH\cdot IJ\cdot KE}{BC\cdot
  DE\cdot FG\cdot HI\cdot JK\cdot EA}$$ is called (by Eves) an
``h-expression,'' meaning that each point $A,\ldots,K$ occurs equally
many times in the numerator and denominator (for example, $E$ occurs
twice), and each line defined by the two lists of six segments also
occurs equally many times (for example,
$\overleftrightarrow{FG}=\overleftrightarrow{HI}$ occurs twice in the
numerator and twice in the denominator).  The statement of Eves'
Theorem is that the value of an h-expression is an invariant under
projective transformations of the plane.  Related identities for
products of distances have been known in projective geometry since at
least \cite{p}, but it is convenient to attribute the above
formulation to Eves.

The notion of h-expression can also be more visually conveyed in terms
of {\underline{coloring}} the configuration --- an idea demonstrated
at a 2011 talk by Marc Frantz \cite{f2}.  Each point in the
configuration of Figure \ref{fig1} is an endpoint of an equal number
of red and black segments, and, dually, each line contains an equal
number of red and black segments.  Then the ratio of the product of
red lengths to the product of black lengths is Eves' invariant.

Eves' Theorem, when stated in a purely projective way (using
homogeneous coordinates, not Euclidean distances, as in Example
\ref{ex4.5}) is itself a special case of a family of invariant ratios
of products of determinants of homogeneous coordinates for points in
projective space over a field $\fd$.  These ratios were well-known in
$19^{th}$ century Invariant Theory (\cite{b}, \cite{c}, \cite{s}), but
have been more recently used (and, sometimes, re-discovered) in
projective geometry applied to computational topics such as vision and
photogrammetry, or automated proofs (\cite{bb}, \cite{crg}, \cite{f},
\cite{rg}).

Eves' Theorem can be stated in terms of a function, where the input is
a configuration of points $\SSS$ in projective space $\fd P^D$, and
the output is a ratio, i.e., an element of the projective line $\fd
P^1$; the content of the Theorem is that the ratio is well-defined
(independent of certain choices made in specifying the configuration)
and also invariant under projective transformations.  Our new
construction, Theorem \ref{thm4.4}, generalizes the target to a
``weighted projective space'' $\fd P(p_0,\ldots,p_n)$, so the
projective line is the special case $\fd P(1,1)$.  In Section
\ref{sec3}, we give a unified treatment of the configurations to which
Theorem \ref{thm4.4} applies, by a weighting, coloring, and indexing
scheme.  A configuration $\SSS$ of points in the projective space $\fd
P^D$ that satisfies a condition (Definition \ref{def4.2}), depending
on the weight vector ${\mathbf p}=(p_0,\ldots,p_n)$, is assigned an
element $E_{\mathbf p}(\SSS)\in\fd P({\mathbf p})$, an invariant under
``morphisms'' of the configuration (Definition \ref{def4.7}), which
generalize projective transformations.  The number of colors is $n+1$,
so the classical case is the assignment of Eves' ratio
$E_{(1,1)}(\SSS)\in\fd P(1,1)$ to some two-color configurations
$\SSS$, and the new weighted invariants apply to a larger category of
multi-color configurations.

In Section \ref{sec1} we review the definition and some elementary
properties of weighted projective spaces --- these properties are
well-known in the complex case, but the real case is different in some
ways we intend to exploit, so we are careful to present all the
necessary details.  Section \ref{sec2} introduces a new notion of
``reconstructibility'' for a weighted projective space; the two main
results are that complex weighted projective spaces are all
reconstructible, and that some real weighted projective spaces are
not.  In Section \ref{sec4} we review a connection between real
projective and Euclidean geometry, and state a Euclidean version of
Theorem \ref{thm4.4}.  Section \ref{sec5} applies the notion of
reconstructibility to show that in the complex case, the weighted
invariant $E_{\mathbf p}$ of a multi-color configuration can be
determined by finding the (classical) $E_{(1,1)}$ ratios for a finite
list of related two-color configurations.  However, in the real case,
Examples \ref{ex6.1}, \ref{ex6.4}, \ref{ex6.5} give pairs of
configurations with different weighted invariants in $\re P({\mathbf
p})$, but which cannot be distinguished by applying the reconstruction
method to the $E_{(1,1)}$ ratios in $\re P^1$.

\section{Weighted projective spaces}\label{sec1}

This Section reviews the definition of weighted projective spaces and
some of their elementary properties.  For the complex case, these
properties (in particular, Examples \ref{ex1.4b} and \ref{ex1.14}),
are well-known (\cite{d}, \cite{dol}); we give elementary proofs with
the intent of showing how the complex and real cases are different.
Only the objects' set-theoretic properties are of interest here, not
their structure as topological or analytic spaces, algebraic
varieties, or orbifolds.  The applications in subsequent Sections use
only $\re$ and $\co$, but to start in a general way, let $\fd$ be any
field.

\subsection{The basic construction}\label{sec1.1}
The ingredients are $n\in\na$, the vector space $\fd^{n+1}$, and a
{\underline{weight}} ${\mathbf p}=(p_0,p_1,\ldots,p_n)\in\na^{n+1}$.
Denote $\fd^{n+1}_*=\fd^{n+1}\setminus\{\mathbf{0}\}$, and for
elements ${\mathbf z}=(z_0,\ldots,z_n)$, $\mathbf
w=(w_0,\ldots,w_n)\in\fd^{n+1}_*$, define a relation $\sim_{\mathbf
  p}$ so that ${\mathbf z}\sim_{\mathbf p}{\mathbf w}$ means there
exists $\lambda\in\fd^1_*$ such
that: $$w_0=\lambda^{p_0}z_0,\ \ w_1=\lambda^{p_1}z_1, \ \ldots,
\ w_n=\lambda^{p_n}z_n.$$ This is an equivalence relation on
$\fd^{n+1}_*$ because $\fd$ is a field.

\begin{defn}\label{def1.1}
  Let $\fd P({\mathbf p})$ denote the set of equivalence classes for
  $\sim_{\mathbf p}$.  $\fd P({\mathbf p})=\fd P(p_0,\ldots,p_n)$ is
  the {\underline{weighted projective space}} corresponding to the
  weight $\mathbf p$.
\end{defn}
\begin{notation}\label{not1.2}
  Let $\pi_{\mathbf p}:\fd^{n+1}_*\to\fd P({\mathbf p})$ denote the
  canonical quotient map, defined so that $\pi_{\mathbf p}({\mathbf
  z})$ is the equivalence class of ${\mathbf z}$.  It is convenient to
  use the same letter for elements of the weighted projective space,
  and square brackets for weighted homogeneous coordinates:
  $$\pi_{\mathbf p}({\mathbf z})=z=[z_0:z_1:\ldots:z_n]_{\mathbf p}.$$
\end{notation}
\begin{example}\label{ex1.2}
  For ${\mathbf p}=(1,1,\ldots,1)$, $\fd P(1,1,\ldots,1)$ is the usual
  projective space, denoted $\fd P^n$, with homogeneous coordinates
  $\pi:(z_0,\ldots,z_n)\mapsto[z_0:\ldots:z_n]$ (omitting the
  subscripts).
\end{example}
\begin{example}\label{ex1.4b}
  For $\fd=\co$ and ${\mathbf p}=(p,p,\ldots,p)$, ${\mathbf z}$ and
  ${\mathbf w}\in\co^{n+1}_*$ are $\sim_{\mathbf p}$-equivalent if and
  only if they are related by non-zero complex scalar multiplication,
  so the following sets are exactly equal: $\co P(p,p,\ldots,p)=\co
  P(1,1,\ldots, 1)=\cpn$.
\end{example}
\begin{example}\label{ex1.6}
  For $\fd=\re$ and ${\mathbf p}=(2k+1,2k+1,\ldots,2k+1)$, ${\mathbf
  z}$ and ${\mathbf w}\in\re^{n+1}_*$ are $\sim_{\mathbf
  p}$-equivalent if and only if they are related by non-zero real
  scalar multiplication, so the following sets are exactly equal: $\re
  P(2k+1,2k+1,\ldots,2k+1)=\re P(1,1,\ldots, 1)=\rpn$.
\end{example}
\begin{example}\label{ex1.4}
  For $\fd=\re$ and ${\mathbf p}=(2k,2k,\ldots,2k)$, the restriction
  of $\pi_{\mathbf p}$ to the unit sphere $S^n\subseteq\re^{n+1}_*$ is
  a one-to-one function onto $\re P({\mathbf p})$.  It is not
  inconvenient to identify the sets: $\re P(2k,2k,\ldots,2k)=S^n$.
\end{example}

\subsection{Mappings}\label{sec1.2}

Let $\fd$ and $\ld$ be fields, and let ${\mathbf p}\in\na^{n+1}$,
${\mathbf q}\in\na^{N+1}$ be weights.  Consider any function
$\mathbf{f}:\fd^{n+1}_*\to\ld^{N+1}$.  Given
$\mathbf{z}\in\fd^{n+1}_*$, suppose $\mathbf{f}$ has the following two
properties: first,
\begin{equation}\label{eq1}
  \mathbf{f}(\mathbf{z})=(w_0,w_1,\ldots,w_N)\ne\mathbf{0},
\end{equation}
and second, for any $\lambda\in\fd^1_*$, there exists
$\mu\in\ld^1_*$ so that
\begin{equation}\label{eq2}
  \mathbf{f}(\lambda^{p_0}z_0,\lambda^{p_1}z_1,\ldots,\lambda^{p_n}z_n)=(\mu^{q_0}w_0,\mu^{q_1}w_1,\ldots,\mu^{q_N}w_N).
\end{equation}
 Then $\mathbf{f}$ also has these two properties at every point
 ${\mathbf z^\prime}\in\fd^{n+1}_*$ in the same equivalence class as
 ${\mathbf z}$, and if ${\mathbf z^\prime}\sim_{\mathbf p}{\mathbf
 z}$, then ${\mathbf f}({\mathbf z^\prime})\sim_{\mathbf q}{\mathbf
 f}({\mathbf z})$.  Let ${\mathbf U}\subseteq\fd^{n+1}_*$ be the set
 of points where $\mathbf{f}$ has the two properties, and let
 $U=\pi_{\mathbf p}({\mathbf U})$.  Then we say ``$\mathbf{f}$ induces
 a map from $\fd P({\mathbf p})$ to $\ld P({\mathbf q})$ which is
 well-defined on the set $U$,'' and denote the induced map, which
 takes $\pi_{\mathbf p}(\mathbf{z})\in U$ to $\pi_{\mathbf q}({\mathbf
 f}({\mathbf z}))$, by $f:z\mapsto f(z)$.  For $z\notin U$, $f(z)$ is
 undefined.
\begin{lem}\label{lem1.5}
  For $\mathbf f$, $f$, and $\mathbf U$ as above, and an element
  $w\in\ld P({\mathbf q})$, let ${\mathbf w}\in\ld^{N+1}_*$ be any
  representative ${\mathbf w}\in w=\pi_{\mathbf q}({\mathbf w})$.
  Then, the inverse image of $w$ is:
  \begin{equation}\label{eq3}
    f^{-1}(w)=\pi_{\mathbf p}(\{{\mathbf z}\in{\mathbf U}:{\mathbf
  f}({\mathbf z})\sim_{\mathbf q}{\mathbf w}\}).
  \end{equation}
\end{lem}
\begin{proof}
  The inverse image is $$f^{-1}(w)=\{z\in\fd P({\mathbf
  p}):z\in\pi_{\mathbf p}({\mathbf U})\mbox{ and }f(z)=w\}.$$ The
  first condition is that $\exists{\mathbf x}\in{\mathbf
  U}:z=\pi_{\mathbf p}({\mathbf x})$, and the second condition is that
  the $\sim_{\mathbf q}$-equivalence class of $\mathbf w$ is the same
  as the $\sim_{\mathbf q}$-equivalence class of ${\mathbf f}({\mathbf
  z})$ for some ${\mathbf z}\in z$.  So, $$f^{-1}(w)=\left\{z\in\fd
  P({\mathbf p}):\left(\exists{\mathbf x}\in{\mathbf U}:{\mathbf x}\in
  z\right)\mbox{ and }\left(\exists{\mathbf z}\in z:{\mathbf
  f}({\mathbf z})\sim_{\mathbf q}{\mathbf w}\right)\right\}.$$ From (\ref{eq3}),
  denote the RHS set (depending on $w$ but not the choice of $\mathbf w$): 
  \begin{eqnarray*}
    A_w&=&\pi_{\mathbf p}(\{{\mathbf z}\in{\mathbf
  U}:{\mathbf f}({\mathbf z})\sim_{\mathbf q}{\mathbf w}\})\\
    &=&\left\{z\in\fd P({\mathbf
  p}):\exists{\mathbf z}\in{\mathbf U}:\left(\pi_{\mathbf p}({\mathbf z})=z\mbox{ and }{\mathbf f}({\mathbf z})\sim_{\mathbf q}{\mathbf w}\right)\right\}.
  \end{eqnarray*}   If $z\in A_w$, letting ${\mathbf
  x}={\mathbf z}$ shows $z\in f^{-1}(w)$.  Conversely, if $z\in
  f^{-1}(w)$, then $\exists{\mathbf x}\in{\mathbf U}:{\mathbf x}\in z$
  and $\exists{\mathbf z}\in z:{\mathbf f}({\mathbf z})\sim_{\mathbf q}{\mathbf w}$.
  Since ${\mathbf x}\in{\mathbf U}$ has properties (\ref{eq1}) and
  (\ref{eq2}), and ${\mathbf z}\sim_{\mathbf p}{\mathbf x}\in z$,
  $\mathbf z$ also has the two properties, so ${\mathbf z}\in{\mathbf
  U}$, and $z\in A_w$.
\end{proof}
Similar reasoning with the above data leads to the following
equivalences:
\begin{prop}\label{prop1.6}
  A map $f:\fd P({\mathbf p})\to\ld P({\mathbf q})$ which is
  well-defined on the set $U$ is an onto map if and only if: for all
  ${\mathbf w}\in\ld^{N+1}_*$, there exists ${\mathbf z}\in{\mathbf
  U}$ such that ${\mathbf f}({\mathbf z})\sim_{\mathbf q}{\mathbf w}$.
  \boxx
\end{prop}
\begin{prop}\label{prop1.7}
  A map $f:\fd P({\mathbf p})\to\ld P({\mathbf q})$ which is
  well-defined on the set $U$ is a one-to-one map if and only if: for
  all ${\mathbf z}$, ${\mathbf z^\prime}\in{\mathbf U}$, if ${\mathbf
  f}({\mathbf z})\sim_{\mathbf q}{\mathbf f}({\mathbf z^\prime})$,
  then ${\mathbf z}\sim_{\mathbf p}{\mathbf z^\prime}$.  \boxx
\end{prop}
\begin{example}\label{ex1.10}
  For $m\in\na$, consider two weights, ${\mathbf
  q}=(q_0,q_1,\ldots,q_n)$ and ${\mathbf p}=(mq_0,mq_1,\ldots,mq_n)$,
  and let ${\mathbf f}$ be the inclusion
  $\fd^{n+1}_*\to\fd^{n+1}:{\mathbf z}\mapsto{\mathbf z}$.  ${\mathbf
  f}$ clearly satisfies (\ref{eq1}) at every point ${\mathbf
  z}\in\fd^{n+1}_*$, and also (\ref{eq2}) with $\mu=\lambda^m$, so
  ${\mathbf U}=\fd^{n+1}_*$.  The induced map $f:\fd P({\mathbf
  p})\to\fd P({\mathbf q})$ is well-defined on $U=\fd P({\mathbf p})$,
  and is an onto map as in Proposition \ref{prop1.6}.  $f$ is also
  one-to-one if it satisfies the condition of Proposition
  \ref{prop1.7}: for all ${\mathbf z}$, ${\mathbf
  z^\prime}\in\fd^{n+1}_*$, if
  $(z_0^\prime,\ldots,z_n^\prime)=(\lambda^{q_0}z_0,\ldots,\lambda^{q_n}z_n)$
  for some $\lambda\in\fd^1_*$, then there exists $\mu\in\fd^1_*$ such
  that
  $(z_0^\prime,\ldots,z_n^\prime)=(\mu^{mq_0}z_0,\ldots,\mu^{mq_n}z_n)$.
  So, if $\fd$ and $m$ have the property that $\forall\lambda\ne0$
  $\exists\mu:\mu^m=\lambda$, then $f$ is one-to-one; for example,
  this happens for $\fd=\co$ and any $m$, and also for $\fd=\re$ and
  odd $m$.  Another situation in which $f$ is one-to-one is the case
  where $\fd=\re$ and all the integers $q_0$, \ldots, $q_n$ are even:
  for any $\lambda\ne0$, let $\mu=|\lambda|^{1/m}$, then for
  $k=0,\ldots,n$,
  $\mu^{mq_k}=(|\lambda|^{1/m})^{mq_k}=|\lambda|^{q_k}=\lambda^{q_k}$.
\end{example}
  In each of the cases mentioned in Example \ref{ex1.10} where the
  induced map $f$ is one-to-one, we have ${\mathbf z}\sim_{\mathbf
  p}{\mathbf z^\prime}\iff{\mathbf z}\sim_{\mathbf q}{\mathbf
  z^\prime}$, so the equivalence classes are the same, $f$ is the
  identity map, and these sets are equal: $\fd P({\mathbf p})=\fd
  P({\mathbf q})$.
\begin{example}\label{ex1.11b}
  The inclusion ${\mathbf f}:\re^{n+1}_*\to\re^{n+1}$ as in Example
  \ref{ex1.10} induces a well-defined, onto map $f:\re
  P(2,2,\ldots,2)\to\re P(1,1,\ldots,1)$, but $f$ is not one-to-one.
  This map $f$ is the usual two-to-one covering $S^n\to\re P^n$ (the
  ``antipodal identification'').
\end{example}
\begin{thm}\label{thm1.12}
  For any weight ${\mathbf q}=(q_0,\ldots,q_n)$, let ${\mathbf
  p}=(2q_0,\ldots,2q_n)$.  Let $f:\re P({\mathbf p})\to\re P({\mathbf
  q})$ be induced by the inclusion ${\mathbf f}$ as in Example
  \ref{ex1.10}.  If $q_j$ is odd, then the restriction of $f$ to the
  set $\pi_{\mathbf p}(\{{\mathbf z}:z_j\ne0\})$ is two-to-one.
\end{thm}
\begin{proof}
  Take any ${\mathbf w}\in\re^{n+1}_*$ with $w_j\ne0$, and let
  $w=\pi_{\mathbf q}({\mathbf w})$. By Lemma \ref{lem1.5},
  \begin{eqnarray*}
    f^{-1}(w)&=&\pi_{\mathbf p}(\{{\mathbf z}\in\re^{n+1}_*:{\mathbf z}\sim_{\mathbf q}{\mathbf w}\})\\
    &=&\pi_{\mathbf p}(\{(\mu^{q_0}w_0,\ldots,\mu^{q_n}w_n):\mu\in\re^1_*\}).
  \end{eqnarray*}
  Two points $(\mu_1^{q_0}w_0,\ldots,\mu_1^{q_n}w_n)$,
  $(\mu_2^{q_0}w_0,\ldots,\mu_2^{q_n}w_n)$ are $\sim_{\mathbf
  p}$-equivalent if and only if there exists $\lambda\in\re^1_*$ such
  that
  \begin{equation}\label{eq15}
    \mu_1^{q_0}w_0=\lambda^{2q_0}\mu_2^{q_0}w_0, \ldots, \mu_1^{q_n}w_n=\lambda^{2q_n}\mu_2^{q_n}w_n.
  \end{equation}
  For $w_j\ne0$ and $q_j$ odd,
  $\mu_1^{q_j}w_j=\lambda^{2q_j}\mu_2^{q_j}w_j\iff\lambda^2=\mu_1/\mu_2$,
  which is equivalent to the system of equations (\ref{eq15}).  So,
  the two points are $\sim_{\mathbf p}$-equivalent if and only if
  $\mu_1$ and $\mu_2$ have the same sign: there are two $\sim_{\mathbf
  p}$-equivalence classes.
\end{proof}
\begin{example}\label{ex1.13}
  If all the $q_j$ are odd, then $f$ as in Theorem \ref{thm1.12} is
  globally two-to-one.  Example \ref{ex1.11b} is a special case.  An
  example with the $q_j$ not all odd is the map $f:\re P(4,2)\to\re
  P(2,1)$.  For this $f$,
  $f^{-1}([1:z_1]_{(2,1)})=\{[1:z_1]_{(4,2)},[1:-z_1]_{(4,2)}\}$, a
  two-element set for $z_1\ne0$, but a singleton for $z_1=0$.
\end{example}
\begin{example}\label{ex1.11}
  For $m\in\na$, consider two weights, ${\mathbf
  q}=(q_0,q_1,q_2,\ldots,q_n)$ and ${\mathbf
  p}=(q_0,mq_1,mq_2,\ldots,mq_n)$, and let ${\mathbf f}$ be the
  polynomial map $${\mathbf
  f}:\fd^{n+1}_*\to\fd^{n+1}:(z_0,z_1,z_2,\ldots,z_n)\mapsto(z_0^m,z_1,z_2,\ldots,z_n).$$
  ${\mathbf f}$ clearly satisfies (\ref{eq1}) at every point ${\mathbf
  z}\in\fd^{n+1}_*$, and also (\ref{eq2}) with $\mu=\lambda^m$, so
  ${\mathbf U}=\fd^{n+1}_*$.  The induced map $f:\fd P({\mathbf
  p})\to\fd P({\mathbf q})$ is well-defined on $U=\fd P({\mathbf p})$.
  $f$ is an onto map if it satisfies the condition of Proposition
  \ref{prop1.6}: for all ${\mathbf w}=(w_0,\ldots,w_n)\in\fd^{n+1}_*$,
  there exist ${\mathbf z}$ and $\lambda$ such that
  $(w_0,\ldots,w_n)=(\lambda^{q_0}z_0^m,\ldots,\lambda^{q_n}z_n)$.
  For example, if $\fd=\co$, or if $\fd=\re$ and $m$ is odd, then
  given ${\mathbf w}$, one can choose $\lambda=1$, any $z_0$ with
  $z_0^m=w_0$, and $z_k=w_k$ for $k=1,\ldots,n$.  Another situation in
  which $f$ is onto is the case where $\fd=\re$ and $q_0$ is odd: for
  $w_0\ge0$, make the same choices mentioned in the previous case, and
  for $w_0<0$, choose $\lambda=-1$, any $z_0$ with
  $z_0^m=(-1)^{q_0}w_0=|w_0|$, and $z_k=w_k/(-1)^{q_k}$ for
  $k=1,\ldots,n$.
\end{example}
\begin{example}\label{ex1.12}
  The polynomial map ${\mathbf
  f}:\re^2_*\to\re^2:(z_0,z_1)\mapsto(z_0^2,z_1)$ induces a
  well-defined map $f:\re P(2,2)\to\re P(2,1)$ as in Example
  \ref{ex1.11}, but the induced map is not onto.  The point
  $[-1:1]_{\mathbf q}$ is not in the image of $f$; there is no
  $(z_0,z_1)\in\re^2_*$ such that $(z_0^2,z_1)\sim_{\mathbf q}(-1,1)$.
\end{example}
\begin{lem}\label{lem1.14}
  For $w_0\in\co^1_*$, $N$, $P\in\na$, suppose
  $\{\zeta_0,\ldots,\zeta_{N-1}\}$ are the $N$ distinct complex roots
  of the equation $\zeta^N=w_0$.  Then the number of distinct elements
  in the set $\{\zeta_0^P,\ldots,\zeta_{N-1}^P\}$ is $\lcm(P,N)/P$.
\end{lem}
\begin{proof}
  In polar form, $w_0=\rho e^{i\theta}$ for a unique $\rho>0$,
  $\theta\in[0,2\pi)$.  By re-labeling if necessary,
  $\zeta_j^P=\rho^{P/N}e^{i(\theta+2\pi j)P/N}$ for $j=0,\ldots,N-1$.
  Let $j$ be the smallest integer such that $jP/N\in\na$.  It follows
  that $j=\lcm(P,N)/P$, the elements $\zeta_k^P$ are distinct for
  $k=0,\ldots,j-1$, and $\zeta_k^P=\zeta_{k+j}^P$.
\end{proof}
\begin{example}\label{ex1.14}
  Let $\fd=\co$.  Consider a weight where $n$ of the $n+1$ integers
  have a common factor $m$ --- without loss of generality, ${\mathbf
  p}=(q_0,mq_1,mq_2,\ldots,mq_n)$ as in Example \ref{ex1.11}. Suppose
  further that the integers $m$ and $q_0$ are relatively prime (have
  no common factor $>1$) --- this can be achieved without changing the
  set $\co P({\mathbf p})$, by dividing out any common factor as in
  Example \ref{ex1.10}.  Now let ${\mathbf
  q}=(q_0,q_1,q_2,\ldots,q_n)$; then the map ${\mathbf
  f}(z_0,z_1,z_2,\ldots,z_n)=(z_0^m,z_1,z_2,\ldots,z_n)$ from Example
  \ref{ex1.11} induces a well-defined, onto map $f:\co P({\mathbf
  p})\to\co P({\mathbf q})$.  It is also one-to-one: to establish this
  as in Proposition \ref{prop1.7}, we have to show that for any
  ${\mathbf z}$, ${\mathbf z^\prime}\in\co^{n+1}_*$, if there exists
  $\lambda\ne0$ such that
  \begin{equation}\label{eq4}
    (\lambda^{q_0}z_0^m,\lambda^{q_1}z_1,\ldots,\lambda^{q_n}z_n)=((z_0^\prime)^m,z_1^\prime,\ldots,z_n^\prime),
  \end{equation}
  then there exists $\mu\ne0$ so that
  \begin{equation}\label{eq5}
    (\mu^{q_0}z_0,\mu^{mq_1}z_1,\ldots,\mu^{mq_n}z_n)=(z_0^\prime,z_1^\prime,\ldots,z_n^\prime).
  \end{equation}
  The algebra problem is: given $\lambda$, ${\mathbf z}$, ${\mathbf
    z^\prime}$, find $\mu$.  If $z_0^\prime=0$, then $z_0=0$ and we
    can pick any $\mu$ satisfying $\mu^m=\lambda$.  If
    $z_0^\prime\ne0$, then $z_0\ne0$, and there are $m$ different
    roots $\{\mu_k:k=0,\ldots,m-1\}$ satisfying $\mu_k^m=\lambda$.
    For $j=1,\ldots,n$, each $\mu_k$ satisfies
    $\mu_k^{mq_j}z_j=\lambda^{q_j}z_j=z_j^\prime$.  Each $\mu_k$ also
    satisfies
  $$(\mu_k^{q_0})^m=\mu_k^{mq_0}=\lambda^{q_0}=(z_0^\prime)^m/z_0^m=(z_0^\prime/z_0)^m,$$
  so each element of the set
  $R_1=\{\mu_0^{q_0},\ldots,\mu_{m-1}^{q_0}\}$ is also one of the $m$
  elements of the set $R_2=\{\xi:\xi^m=(z_0^\prime/z_0)^m\}$.  One of
  the elements of $R_2$ is $z_0^\prime/z_0$.  Using the assumption
  that $m$ and $q_0$ are relatively prime and Lemma \ref{lem1.14},
  $R_1$ has $m$ distinct elements, so there is some $k$ such that
  $\mu_k^{q_0}=z_0^\prime/z_0$.  This $\mu_k$ is the $\mu$ required in
  (\ref{eq5}), to show $f$ is one-to-one.
\end{example}

\begin{example}\label{ex1.15}
  Let $\fd=\re$, and consider, as in Example \ref{ex1.14}, a weight
  ${\mathbf p}=(q_0,mq_1,mq_2,\ldots,mq_n).$ Here we assume $m$ is odd
  but make no assumption on $q_0$.  Now let ${\mathbf
    q}=(q_0,q_1,q_2,\ldots,q_n)$; then the map ${\mathbf
    f}(z_0,z_1,z_2,\ldots,z_n)=(z_0^m,z_1,z_2,\ldots,z_n)$ from
  Example \ref{ex1.11} induces a well-defined, onto map $f:\re
  P({\mathbf p})\to\re P({\mathbf q})$.  It is also one-to-one: the
  algebra problem is to solve the same equations (\ref{eq4}),
  (\ref{eq5}), for a real $\mu$ in terms of real ${\mathbf z}$,
  ${\mathbf z^\prime}$, $\lambda$.  Given $\lambda\ne0$, let $\mu$ be
  the unique real solution of $\mu^m=\lambda$.  Then, for
  $j=1,\ldots,n$, $\mu^{mq_j}z_j=\lambda^{q_j}z_j=z_j^\prime$, and
  $(\mu^{q_0}z_0)^m=\lambda^{q_0}z_0^m=(z_0^\prime)^m\implies\mu^{q_0}z_0=z_0^\prime$.
\end{example}

\begin{example}\label{ex1.16}
  Let $\fd=\re$, and consider, as in Example \ref{ex1.15}, a weight
  ${\mathbf p}=(q_0,mq_1,mq_2,\ldots,mq_n).$ Here we assume $m$ is
  even, $q_0$ is odd, and all $q_1,\ldots,q_n$ are even.  Now let
  ${\mathbf q}=(q_0,q_1,q_2,\ldots,q_n)$; then the map ${\mathbf
    f}(z_0,z_1,z_2,\ldots,z_n)=(z_0^m,z_1,z_2,\ldots,z_n)$ from
  Example \ref{ex1.11} induces a well-defined, onto map $f:\re
  P({\mathbf p})\to\re P({\mathbf q})$.  It is also one-to-one: the
  algebra problem is to solve (\ref{eq4}), (\ref{eq5}), for a real
  $\mu$ in terms of real ${\mathbf z}$, ${\mathbf z^\prime}$,
  $\lambda$.  Given $\lambda\ne0$, the equation $\mu^m=|\lambda|$ has
  exactly two real solutions,
  $\{\mu_1=|\lambda|^{1/m},\mu_2=-|\lambda|^{1/m}\}$.  Then, for
  $k=1,2$, $j=1,\ldots,n$,
  $\mu_k^{mq_j}z_j=|\lambda|^{q_j}z_j=\lambda^{q_j}z_j=z_j^\prime$.
  For $k=1,2$,
  $(\mu_k^{q_0}z_0)^m=|\lambda|^{q_0}z_0^m=|\lambda^{q_0}z_0^m|=(z_0^\prime)^m$,
  so the set $\{\mu_1^{q_0}z_0,\mu_2^{q_0}z_0=-\mu_1^{q_0}z_0\}$ is
  contained in the set $\{z_0^\prime,-z_0^\prime\}$, and one of the
  two roots is the required $\mu$ satisfying
  $\mu^{q_0}z_0=z_0^\prime$.
\end{example}
\begin{example}\label{ex1.17}
  For an even number $p_1$, the function ${\mathbf
  f}(z_0,z_1)=(z_0^{p_1},z_1)$ induces a well-defined, onto map $f:\re
  P(1,p_1)\to\re P(1,1)$ as in Example \ref{ex1.11}.  The induced map
  is not one-to-one: ${\mathbf f}(0,1)=(0,1)\sim_{\mathbf q}{\mathbf
  f}(0,-1)=(0,-1)$, but $(0,1)\not\sim_{\mathbf p}(0,-1)$.
\end{example}

\section{Reconstructibility}\label{sec2}

Given a weight ${\mathbf p}$, and two indices $i<j$ in
$\{0,1,\ldots,n\}$, consider numbers $a_{ij}$, $b_{ij}\in\na$ and a
mapping ${\mathbf c}_{ij}:\fd^{n+1}_*\to\fd^2$ defined by the formula
$${\mathbf
c}_{ij}(z_0,z_1,\ldots,z_i,\ldots,z_j,\ldots,z_n)=(z_i^{a_{ij}},z_j^{b_{ij}}).$$The
function ${\mathbf c}_{ij}$ satisfies (\ref{eq1}) on the complement of
the set $\{{\mathbf z}:z_i=z_j=0\}$, and if the products are equal:
$p_ia_{ij}=p_jb_{ij}$, then it also satisfies (\ref{eq2}) for weights
${\mathbf p}$ and ${\mathbf q}=(1,1)$.
\begin{notation}\label{not1.18}
  For ${\mathbf p}$, $a_{ij}$, $b_{ij}$ as above, the function
  ${\mathbf c}_{ij}$ induces a map
  \begin{eqnarray*}
    c_{ij}:\fd P({\mathbf p})&\to&\fd P^1:\\
    \left[z_0:z_1:\ldots:z_i:\ldots:z_j:\ldots:z_n\right]_{\mathbf
    p}&\mapsto&\left[z_i^{a_{ij}}:z_j^{b_{ij}}\right],
  \end{eqnarray*}
  which is well-defined on the complement of the set
  $\{[z_0:\ldots:z_n]_{\mathbf p}:z_i=z_j=0\}$.  We call such a map an
  {\underline{axis projection}}.
\end{notation}
\begin{lem}\label{lem1.19}
  Given a weight $\mathbf p$ and indices $i$, $j$, let
  $\ell_{ij}=\lcm(p_i,p_j)$.  Then 
  \begin{equation}\label{eq6}
    h_{ij}:\fd P({\mathbf p})\to\fd
  P^1:z\mapsto[z_i^{\ell_{ij}/p_i}:z_j^{\ell_{ij}/p_j}]
  \end{equation} is an axis
  projection.  For any axis projection $c_{ij}$ as in Notation
  \ref{not1.18}, there exists $k_{ij}\in\na$ such that $c_{ij}$
  factors as $c_{ij}=G_{ij}\circ h_{ij}$, where the function
  $$G_{ij}:\fd P^1\to\fd
  P^1:[w_0:w_1]\mapsto[w_0^{k_{ij}}:w_1^{k_{ij}}]$$ is well-defined on
  $\fd P^1$.
\end{lem}
\begin{proof}
  $\ell_{ij}$ is the least common multiple of $p_i$ and $p_j$.  By
  elementary number theory (\cite{o} Ch.\ 3), any other common
  multiple is divisible by $\ell_{ij}$, so there exists $k_{ij}$ so
  that $p_ia_{ij}=p_jb_{ij}=k_{ij}\ell_{ij}$.
\end{proof}
\begin{notation}\label{not1.20}
  Let $I$ be the set of index pairs $\{(i,j):0\le i<j\le n\}$.  Let
  $D_{\mathbf p}\subseteq\fd P({\mathbf p})$ be the set of points
  where all the coordinates are non-zero: $\{z_0\ne0,\ z_1\ne0,\
  \ldots,$ and $z_n\ne0\}$.  Given axis projections $c_{ij}$ for
  $(i,j)\in I$, let $\displaystyle{\prod c_{ij}}$ denote the map
  \begin{eqnarray*}
    \fd P({\mathbf p})&\to&\fd P^1\times\fd P^1\times\cdots\times\fd P^1\\
    z&\mapsto&\left(c_{11}(z),c_{12}(z),\ldots,c_{ij}(z),\ldots,c_{n-1,n}(z)\right).
  \end{eqnarray*}
\end{notation}
The target space in the above Notation has one $\fd P^1$ factor for
each of the elements of $I$ ($\#I=n(n+1)/2$), so the output formula
lists an axis projection for every index pair $(i,j)$.  The map
$\displaystyle{\prod c_{ij}}$ is well-defined at every point in $D_{\mathbf p}$,
and possibly at some points not in $D_{\mathbf p}$.
\begin{defn}\label{def1.21}
  A weighted projective space $\fd P({\mathbf p})$ is
  {\underline{reconstructible}} means: there exist axis projections
  such that the restriction of the map $\displaystyle{\prod c_{ij}}$
  to the domain $D_{\mathbf p}$ is one-to-one.
\end{defn}
The idea is to try to use a list of unweighted ratios, $c_{ij}(z)$, as
a coordinatization of the weighted projective space $\fd P({\mathbf
p})$.  A reconstructible space is one where any point $z$ (with
non-zero coordinates) can be uniquely ``reconstructed'' from a list of
its values under some axis projections.  The use of the domain
$D_{\mathbf p}$ in the Definition omits consideration of points $z$
with a zero coordinate; as already seen in Examples \ref{ex1.13} and
\ref{ex1.17}, such points can exhibit exceptional behavior, and we are
interested in properties of generic points.
\begin{lem}\label{lem1.22}
  Given ${\mathbf p}$, the following are equivalent.
  \begin{enumerate}
    \item $\fd P({\mathbf p})$ is reconstructible;
    \item for the axis projections $h_{ij}$ from
    \mbox{\rm{(\ref{eq6})}}, the map $\displaystyle{\prod_{(i,j)\in
    I}h_{ij}}$ is one-to-one on $D_{\mathbf p}$;
    \item there exist a subset $J\subseteq I$ and axis projections
    $c_{ij}$ so that $\displaystyle{\prod_{(i,j)\in J} c_{ij}}$ is
    one-to-one on $D_{\mathbf p}$;
    \item there exists a subset $J\subseteq I$ so that $\displaystyle{\prod_{(i,j)\in J} h_{ij}}$ is
    one-to-one on $D_{\mathbf p}$.
  \end{enumerate}
\end{lem}
\begin{proof}
  The implications ${\mathit 2}\implies{\mathit 1}\implies{\mathit 3}$
  and ${\mathit 2}\implies{\mathit 4}\implies{\mathit 3}$ are
  logically trivial.  To show ${\mathit 3}\implies{\mathit 2}$, given
  $c_{ij}$ for $J\subseteq I$, pick any axis projections $c_{ij}$ for
  the remaining indices not in $J$; then
  $$\displaystyle{\prod_{(i,j)\in J} c_{ij}=F\circ\prod_{(i,j)\in I}
  c_{ij}},$$ where $F:(\fd P^1)^{\#I}\to(\fd P^1)^{\#J}$ forgets
  entries with non-$J$ indices.  Then, by Lemma \ref{lem1.19}, there
  exist factorizations $c_{ij}=G_{ij}\circ h_{ij}$, so
  $$\prod_{(i,j)\in J} c_{ij}=F\circ\left(\prod_{(i,j)\in
  I}G_{ij}\right)\circ\left(\prod_{(i,j)\in I} h_{ij}\right)$$ (where
  the product map $\displaystyle{\prod G_{ij}}:(\fd P^1)^{\#I}\to(\fd
  P^1)^{\#I}$ is defined in the obvious way for the composition to
  make sense).  If $\displaystyle{\prod h_{ij}}$ is not one-to-one on
  $D_{\mathbf p}$, then $\displaystyle{\prod_{(i,j)\in J} c_{ij}}$ is also not
  one-to-one on $D_{\mathbf p}$.
\end{proof}
\begin{example}\label{ex1.24b}
  For any field $\fd$, the space $\fd P(1,p_1,\ldots,p_n)$ is
  reconstructible.  Only $n$ axis projections are needed for a
  one-to-one product map: let $J=\{(0,j):j=1,\ldots,n\}$, and consider
  ${\mathbf h}_{0j}(z_0,z_1,\ldots,z_n)=(z_0^{p_j},z_j)$.  If
  $z=\pi_{\mathbf p}({\mathbf z})$, $z^\prime=\pi_{\mathbf p}({\mathbf
  z^\prime})\in D_{\mathbf p}$ satisfy ${\mathbf h}_{0j}({\mathbf
  z})\sim_{(1,1)}{\mathbf h}_{0j}({\mathbf z^\prime})$ for
  $j=1,\ldots,n$, then there exist $\lambda_{0j}\ne0$ such that
  $(z_0^\prime)^{p_j}=\lambda_{0j}z_0^{p_j}$ and
  $z_j^\prime=\lambda_{0j}z_j$.  Let $\mu=z_0^\prime/z_0$ (using the
  assumption that $z_0\ne0$); then $\mu z_0=z_0^\prime$ and
  $\mu^{p_j}z_j=(z_0^\prime/z_0)^{p_j}z_j=\lambda_{0j}z_j=z_j^\prime$,
  so ${\mathbf z}\sim_{\mathbf p}{\mathbf z^\prime}$.
\end{example}
\begin{example}\label{ex1.24}
  Example \ref{ex1.24b} shows $\re P(1,p_1)$ is reconstructible.  Even
  though the map $h_{01}([z_0:z_1]_{\mathbf p})=[z_0^{p_1}:z_1]$ is
  not globally one-to-one when $p_1$ is even, as shown in Example
  \ref{ex1.17}, it is one-to-one when restricted to $D_{\mathbf p}$.
\end{example}
\begin{example}\label{ex1.25}
    If one of the numbers $p_0$, $p_1$ is odd, then $\re P(p_0,p_1)$
    is reconstructible.  WLOG, let $p_0$ be odd.  For the axis
    projection $c_{01}([z_0:z_1]_{\mathbf p})=[z_0^{p_1}:z_1^{p_0}]$,
    the following diagram is commutative.  The label on the left arrow
    means that the indicated map is induced by the polynomial map
    $\re^2_*\to\re^2:(z_0,z_1)\mapsto(z_0,z_1^{p_0})$.
  $$\begin{diagram}
    \re P(p_0,p_1)\ar[rrr]^{c_{01}}\ar[d]_{(z_0,z_1^{p_0})}&&&\re P^1\\
    \re P(1,p_1)\ar[urrr]_{(z_0^{p_1},z_1)}&&&
  \end{diagram}$$
  The map on the left is globally one-to-one as in Example
  \ref{ex1.15}, and takes $D_{(p_0,p_1)}$ to $D_{(1,p_1)}$.  The lower
  right map is one-to-one on $D_{(1,p_1)}$: either by Example
  \ref{ex1.15} for odd $p_1$, or by Example \ref{ex1.24} for even
  $p_1$.
\end{example}
\begin{example}
  If both $p_0$ and $p_1$ are even, then $\re P(p_0,p_1)$ is not
  reconstructible.  Consider an axis projection induced by ${\mathbf
  c}_{01}(z_0,z_1)=(z_0^a,z_1^b)$.  By Lemma \ref{lem1.22}, we may
  assume that $a$ and $b$ are not both even.  If $a$ and $b$ are both
  odd, then ${\mathbf c}_{01}(1,1)=(1,1)\sim_{(1,1)}{\mathbf
  c}_{01}(-1,-1)=(-1,-1)$, but $(1,1)\not\sim_{\mathbf p}(-1,-1)$, so
  $c_{01}$ is not one-to-one.  If $a$ is even and $b$ is odd (the
  remaining case being similar), then ${\mathbf
  c}_{01}(1,-1)=(1,-1)\sim_{(1,1)}{\mathbf c}_{01}(-1,-1)=(1,-1)$, but
  $(1,-1)\not\sim_{\mathbf p}(-1,-1)$, so again $c_{01}$ is not
  one-to-one.
\end{example}
\begin{thm}\label{thm1.28}
  For $\fd=\co$ and any weight $\mathbf p$, $\co P({\mathbf p})$ is
  reconstructible.
\end{thm}
\begin{proof}
  Case 1: $n=1$.  Any space $\co P(q_0,q_1)$ is reconstructible (in
  fact, a stronger result holds: there is an axis projection which is
  one-to-one on the entire space).  Let $g_{01}=\gcd(q_0,q_1)$ and
  $\ell_{01}=\lcm(q_0,q_1)$, so $q_0=g_{01}p_0$, $q_1=g_{01}p_1$, and
  $\ell_{01}=g_{01}p_0p_1$, where $p_0$, $p_1$ are relatively prime.  The
  map $${\mathbf
  h}_{01}:\co^2_*\to\co^2:(z_0,z_1)\mapsto(z_0^{\ell_{01}/q_0},z_1^{\ell_{01}/q_1})=(z_0^{p_1},z_1^{p_0})$$
  induces an axis projection $h_{01}$ as in (\ref{eq6}), so that the
  following diagram is commutative.
  $$\begin{diagram}
    \co P(q_0,q_1)\ar[rrr]^{h_{01}}\ar[d]_{Id}&&&\co P^1\\
    \co P(p_0,p_1)\ar[rrr]_{(z_0^{p_1},z_1)}&&&\co P(p_0,1)\ar[u]_{(z_0,z_1^{p_0})}
  \end{diagram}$$
  The left arrow, labeled $Id$, represents the identity map as in
  Example \ref{ex1.10} with $m=g_{01}$.  The map indicated by the
  lower arrow is induced by the polynomial map
  $\co^2_*\to\co^2:(z_0,z_1)\mapsto(z_0^{p_1},z_1)$.  Both maps,
  indicated by the lower and right arrows, are (globally) one-to-one
  as in Example \ref{ex1.14}, so we can conclude $h_{01}$ is
  one-to-one on the entire domain $\co P(q_0,q_1)$.

  Case 2: $n>1$.  We use a product of axis projections as in statement
  $\mathit 2$.\ from Lemma \ref{lem1.22}.  For $(i,j)\in I$, recall
  the notation $\ell_{ij}=\lcm(p_i,p_j)$, and fix
  \begin{equation}\label{eq20}
   a_{ij}=\ell_{ij}/p_i, \ \ \  b_{ij}=\ell_{ij}/p_j,
  \end{equation} and
  $g_{ij}=\gcd(p_i,p_j)$.  Consider the product map
  \begin{equation}\label{eq13}
    \prod_{(i,j)\in I} h_{ij}:[z_0:\ldots:z_n]_{\mathbf
    p}\mapsto\prod_{(i,j)\in I}[z_i^{a_{ij}}:z_j^{b_{ij}}].
  \end{equation}
  To show this product map is one-to-one on $D_{\mathbf p}$, suppose
  we are given ${\mathbf z}$, ${\mathbf z^\prime}$ (with no zero
  components), and constants $\lambda_{ij}\in\co^1_*$ such that
  $\lambda_{ij}(z_i^\prime)^{a_{ij}}=z_i^{a_{ij}}$ and
  $\lambda_{ij}(z_j^\prime)^{b_{ij}}=z_j^{b_{ij}}$.  The algebra
  problem is then to find $\mu\in\co^1_*$ such that
  \begin{equation}\label{eq7}
    \mu^{p_j}z_j^\prime=z_j\mbox{\ \ \ for $j=0,\ldots,n$.}
  \end{equation}
  There are $p_0$ distinct elements $\{\mu_k,k=0,\ldots,p_0-1\}$
  satisfying $\mu_k^{p_0}=z_0/z_0^\prime$.  For each $k$ and for any
  $j=1,\ldots,n$,
  \begin{eqnarray}
    (\mu_k^{p_j}z_j^\prime)^{b_{0j}}&=&\mu_k^{p_jb_{0j}}(z_j^\prime)^{b_{0j}}=\mu_k^{p_0a_{0j}}(z_j^\prime)^{b_{0j}}\nonumber\\
    &=&\left(\frac{z_0}{z_0^\prime}\right)^{a_{0j}}(z_j^\prime)^{b_{0j}}=\lambda_{0j}(z_j^\prime)^{b_{0j}}=z_j^{b_{0j}}.\label{eq14}
  \end{eqnarray}
  By Lemma \ref{lem1.14},
  $$\#\{\mu_k^{p_j}z_j^\prime:k=0,\ldots,p_0-1\}=\#\{\mu_k^{p_j}\}=\frac{\lcm(p_0,p_j)}{p_j}=\frac{\ell_{0j}}{p_j}=b_{0j},$$
  which is equal to the number of roots in
  $\{\xi:\xi^{b_{0j}}=z_j^{b_{0j}}\}$, and so for each $j=1,\ldots,n$,
  there exists some index $k_j$ such that
  $\mu_{k_j}^{p_j}z_j^\prime=z_j$.  At this point we note that if all
  the $k_1,\ldots,k_n$ index values were the same, $\mu=\mu_{k_j}$
  would satisfy (\ref{eq7}) and we would be done.  One case where this
  happens in a trivial way is $p_0=1$; this was already observed in
  Example \ref{ex1.24b}.

  The rest of the Proof does not attempt to show the $k_j$ values are
  equal to each other; instead we use their existence to establish the
  existence of some other index $x$ such that $\mu_x$ is the required
  solution of (\ref{eq7}).

  For $i,j=1,\ldots,n$ with $i<j$, $\mu_{k_i}$ and $\mu_{k_j}$
  satisfy:
  \begin{eqnarray}
    (\mu_{k_i}^{p_i}z_i^\prime)^{a_{ij}}&=&\mu_{k_i}^{\ell_{ij}}(z_i^\prime)^{a_{ij}}=z_i^{a_{ij}}=\lambda_{ij}(z_i^\prime)^{a_{ij}},\nonumber\\
    (\mu_{k_j}^{p_j}z_j^\prime)^{b_{ij}}&=&\mu_{k_j}^{\ell_{ij}}(z_j^\prime)^{b_{ij}}=z_j^{b_{ij}}=\lambda_{ij}(z_j^\prime)^{b_{ij}}\nonumber\\
    \implies \lambda_{ij}&=&\mu_{k_i}^{\ell_{ij}}=\mu_{k_j}^{\ell_{ij}}.\label{eq8}
  \end{eqnarray}
  By re-labeling the roots if necessary, as in the Proof of Lemma
  \ref{lem1.14}, we may assume $\mu_k=r^{1/p_0}e^{i(\theta+2\pi
  k)/p_0}$ for $k=0,\ldots,p_0-1$.  Then (\ref{eq8}) implies the
  congruence 
  \begin{equation}\label{eq12}
    k_j\ell_{ij}\equiv k_i\ell_{ij}\mod p_0.
  \end{equation}
  We are looking for an index $x$ such that for every $j=1,\ldots,n$,
  $\mu_x^{p_j}=\mu_{k_j}^{p_j}$, so $x$ is an integer solution to the
  following system of linear congruences, where $p_j$ and $k_j$ are
  known:
  \begin{equation}\label{eq9}
    xp_j\equiv k_jp_j\mod p_0\mbox{\ \ \ for $j=1,\ldots,n$.}\\
  \end{equation}
  Dividing each congruence by $\gcd(p_0,p_j)$ does not change the
  solution set:
  \begin{eqnarray*}
    \frac{xp_j}{g_{0j}}&\equiv&\frac{k_jp_j}{g_{0j}}\mod\frac{p_0}{g_{0j}}\\
    \iff a_{0j}x&\equiv&a_{0j}k_j\mod b_{0j},
  \end{eqnarray*}
  which is equivalent, since $a_{0j}$ and $b_{0j}$ are relatively prime, to:
  \begin{equation}\label{eq10}
    x\equiv k_j\mod b_{0j}.
  \end{equation}
  By (an elementary generalization of) the Chinese Remainder Theorem
  (\cite{o}, Thm.\ 10--4), there exists an integer solution $x$ of the
  system (\ref{eq10}) if and only if for all pairs $1\le i<j\le n$,
  \begin{equation}\label{eq11}
    k_i\equiv k_j\mod\gcd(b_{0i},b_{0j}).
  \end{equation}
  Property (\ref{eq11}) follows from (\ref{eq12}): each congruence
  (\ref{eq11}) is equivalent to $$k_i\ell_{ij}\equiv
  k_j\ell_{ij}\mod\gcd(b_{0i},b_{0j})\ell_{ij}.$$ The following
  equalities are elementary (\cite{o} Chs.\ 3, 5); one step uses the
  property that $a_{ij}$ and $b_{ij}$ are relatively prime:
  \begin{eqnarray*}
    \gcd(b_{0i},b_{0j})\ell_{ij}&=&\gcd(b_{0i}\ell_{ij},b_{0j}\ell_{ij})=\gcd(b_{0i}p_ia_{ij},b_{0j}p_jb_{ij})\\
    &=&\gcd(\ell_{0i}a_{ij},\ell_{0j}b_{ij})=\gcd(\ell_{0i},\ell_{0j})\\
    &=&\gcd(\lcm(p_0,p_i),\lcm(p_0,p_j))=\lcm(p_0,\gcd(p_i,p_j))\\
    &=&\lcm(p_0,g_{ij}).
  \end{eqnarray*}
  By definition, $\ell_{ij}$ is a multiple of $g_{ij}$, and by
  (\ref{eq12}), $k_i\ell_{ij}-k_j\ell_{ij}$ is a multiple of $p_0$.  It
  follows that $k_i\ell_{ij}-k_j\ell_{ij}$ is a common multiple of
  $p_0$ and $g_{ij}$, and so a multiple of $\lcm(p_0,g_{ij})$, which
  implies (\ref{eq11}).
\end{proof}
\begin{thm}\label{thm1.29}
  For $\fd=\re$, $\re P({\mathbf p})$ is reconstructible if and only
  if $p_0,\ldots,p_n$ are not all even.
\end{thm}
\begin{proof}
  To establish reconstructibility, assume, WLOG, $p_0$ is odd.  We can
  proceed with the same notation as Case 2 of the Proof of Theorem
  \ref{thm1.28}, and use a product of axis projections as in
  (\ref{eq13}), although as in Example \ref{ex1.24b}, only $n$ axis
  projections, indexed by $(i,j)=(0,j)$, are needed for a one-to-one
  product map.  Given real $\mathbf z$, ${\mathbf z^\prime}$, and
  $\lambda_{0j}$, the algebra problem is to find a real solution $\mu$
  of Equation (\ref{eq7}).  Since $p_0$ is odd and $z_0\ne0$, the
  equation $\mu^{p_0}z_0^\prime=z_0$ has a unique real solution for
  $\mu$.  For each $j$, $b_{0j}=p_0/g_{0j}$ is odd, and using the
  unique solution for $\mu$ in Equation (\ref{eq14}) gives
  $(\mu^{p_j}z_j^\prime)^{b_{0j}}=z_j^{b_{0j}}$, which implies
  $\mu^{p_j}z_j^\prime=z_j$, so Equation (\ref{eq7}) is satisfied.

  For the converse, suppose $p_j=2^{e_j}q_j$ with $e_j>0$ and $q_j$
  odd for $j=0,\ldots,n$.  To show statement $\mathit 2.$ from Lemma
  \ref{lem1.22} is false, we show that the product of axis projections
  as in (\ref{eq6}), (\ref{eq13}) is exactly two-to-one on $D_{\mathbf
  p}$; let this map be denoted by the top arrow in the diagram below.
  WLOG, assume $e_0$ is the smallest of the $e_j$ exponents.  By
  Example \ref{ex1.10}, dividing the weight $p$ by $m=2^{e_0-1}$ does
  not change the weighted projective space; this identity map is shown
  as the left arrow in the diagram.
  $$\begin{diagram} \re P({\mathbf
    p})\ar[r]\ar[d]_{Id}&\displaystyle{\prod_{(i,j)\in I}\re P^1}\\
    \re
    P(2q_0,2^{e_1-e_0+1}q_1,\ldots,2^{e_n-e_0+1}q_n)\ar[r]_(0.53){2:1}&\re
    P(q_0,2^{e_1-e_0}q_1,\ldots,2^{e_n-e_0}q_n)\ar[u]
  \end{diagram}$$
  The lower arrow is the map from Theorem \ref{thm1.12}; it is induced
  by the inclusion $\re^{n+1}_*\to\re^{n+1}$, and is two-to-one on the
  set $\{z:z_0\ne0\}$, which contains $D_{\mathbf p}$.  The upward
  arrow on the right is defined as in statement $\mathit 2.$ from
  Lemma \ref{lem1.22}; this was shown to be one-to-one on
  $D_{(q_0,\ldots,2^{e_n-e_0}q_n)}$ in the first part of this Proof.
  The diagram is commutative (the top arrow is the composite of the
  other arrows) because the axis projections use the same exponents.
  For the top arrow,
  \begin{eqnarray*}
    a_{ij}&=&\frac{\lcm(p_i,p_j)}{p_i}=\frac{\lcm(2^{e_i}q_i,2^{e_j}q_j)}{2^{e_i}q_i}\\
    &=&\frac{\lcm(2^{e_i-e_0}q_i,2^{e_j-e_0}q_j)2^{e_0}}{2^{e_i}q_i}.
  \end{eqnarray*}
  For the right arrow, the corresponding exponent is
  $$\frac{\lcm(2^{e_i-e_0}q_i,2^{e_j-e_0}q_j)}{2^{e_i-e_0}q_i},$$
  which is the same, and similarly for the exponents $b_{ij}$.
\end{proof}

\section{Generalizing Eves' Theorem}\label{sec3}

\subsection{Configurations in projective space}\label{sec4.1}

We begin with some combinatorial notation that is needed to keep track
of various indexings.

\begin{notation}\label{not4.0}
  Two ordered $N$-tuples $$(x_1,\ldots,x_N),\ (y_1,\ldots,y_N)$$ are
  {\underline{equivalent up to re-ordering}} if there exists a
  permutation $\sigma$ of the index set $\{1,\ldots,N\}$ such that
  $y_i=x_{\sigma(i)}$ for $i=1,\ldots,N$.  This is an equivalence
  relation; we denote the equivalence class of $(x_1,\ldots,x_N)$ with
  square brackets, $[x_1,\ldots,x_N]$, and call it an
  {\underline{unordered list}}.
\end{notation}
When it is necessary to index the entries in an unordered list, it is
convenient to first pick an ordered representative.
Using the following notation, we describe some configurations of
points in (non-weighted) projective space.
\begin{notation}\label{not4.2}
  Given $D,r\in\na$, and points
  $\alpha_1,\ldots,\alpha_e,\ldots,\alpha_r\in\fd P^D$, denote an
  ordered $r$-tuple of points $$\vec
  s=(\alpha_1,\ldots,\alpha_e,\ldots,\alpha_r).$$ Such an ordered
  $r$-tuple is an {\underline{independent}} $r$-tuple means: there
  exist representatives for the points,
  $\ba_1,\ldots,\ba_e,\ldots,\ba_r$, which form an independent set of
  $r$ vectors in $\fd^{D+1}$ (so $r\le D+1$).  In the case $r=2$, we
  call the ordered, independent pair $\vec s=(\alpha,\beta)$ a
  {\underline{directed segment}}, and the two points its
  {\underline{endpoints}}.  In the case $r=3$, ordered, independent
  triples are {\underline{triangles}} $\vec
  s=\triangle(\alpha\beta\gamma)$.
\end{notation}
\begin{defn}\label{def4.4}
  Given an independent $r$-tuple $\vec s$, there is a unique
  $r$-dimensional subspace ${\mathbf L}$ of $\fd^{D+1}$ that is
  spanned by any independent set of representatives for the points in
  $\vec s$.  The image $\pi({\mathbf L}\setminus\vec0)=L$ is a
  $(r-1)$-dimensional projective subspace of $\fd P^D$, which we call
  the {\underline{span}} of $\vec s$.
\end{defn}
It is convenient to also refer to the $\fd$-linear subspace $\mathbf
L$ as $\pi^{-1}(L)$, and to $L=\pi({\mathbf L})$, even though $\pi$ is
not defined at $\vec0$.
\begin{defn}\label{def4.3}
  Given a weight ${\mathbf p}=(p_0,\ldots,p_n)$ as in Section
  \ref{sec1} and some other numbers $D,\ell,r\in\na$ with $r\le D+1$,
  a {\underline{$({\mathbf p},r,\ell,D)$-configuration}} (or, just
  ``configuration'' when the ${\mathbf p}$, $r$, $\ell$, and $D$ are
  understood) is an ordered $(n+1)$-tuple $\SSS$,
  \begin{equation}\label{eq21}
    \SSS=(\SSS_0,\ldots,\SSS_c,\ldots,\SSS_n),
  \end{equation}
  where each $\SSS_c$ is an unordered list (possibly with repeats) of
  $\ell\cdot p_c$ ordered, independent $r$-tuples of points in $\fd
  P^D$:
  $$\SSS_{c}=[\vec s_c^1,\ldots,\vec s_c^{\ \ell p_c}].$$
\end{defn}
We remark that it is possible for some $\SSS$ to be both a $({\mathbf
p},r,\ell,D)$-configuration and a $({\mathbf
p^\prime},r,\ell^\prime,D)$-configuration with ${\mathbf p}\ne{\mathbf
p^\prime}$ and $\ell\ne\ell^\prime$, although if ${\mathbf p}$ is
given, then $\ell$ is determined by the length of the lists $\SSS_c$.

As an aid to visualization and drawing, the indices $c=0,\ldots,n$ can
correspond to colors: $c=0=$ black, $c=1=$ red, $c=2=$ green, etc.
So, for $r=2$, $\SSS_0$ is a list of $\ell p_0$ black segments,
$\SSS_1$ is a list of $\ell p_1$ red segments, etc. 

\begin{notation}\label{not4.6}
  Given a $({\mathbf p},r,\ell,D)$-configuration $\SSS$, define the
  following sets:
  \begin{itemize}
    \item $\pp(\SSS)$ is the set of points $z\in\fd P^D$
  such that $z$ is one of the $r$ components of some $\vec s_c^K$ in
  $\SSS_c$, for some $c=0,\ldots,n$;
    \item $\LL(\SSS)$ is the set of $(r-1)$-dimensional projective subspaces
  which are the spans of the $r$-tuples $\vec s_c^K$;
    \item $\uuu(\SSS)$ is the following union of $r$-dimensional
    subspaces in $\fd^{D+1}$:
    $$\displaystyle{\uuu(\SSS)=\bigcup_{L\in\LL(\SSS)}\pi^{-1}(L)}.$$
  \end{itemize}
  All three sets depend only on the set of $r$-tuples in $\SSS$, not
  on the ordering in (\ref{eq21}), nor on $\mathbf p$ and $\ell$.
\end{notation}  
For example, when $r=2$, any directed segment lies on a unique
projective line, so $\LL(\SSS)$ is a (finite) set of projective lines
in $\fd P^D$.  Since the same point may appear in several different
$r$-tuples, it is possible for the size of $\pp(\SSS)$ to be small
compared to the number of $r$-tuples.
\begin{defn}\label{def4.1}
  Given a $({\mathbf p},r,\ell,D)$-configuration $\SSS$, for
  $c=0,\ldots,n$, choose an ordered representative 
  \begin{equation}\label{eq37}
    S_c=(\vec s_c^1,\ldots,\vec s_c^K,\ldots,\vec s_c^{\ \ell p_c})
  \end{equation} of the
  equivalence class $\SSS_c$.  Define the {\underline{$c$-degree}} of
  a point $z\in\fd P^D$ to be $$deg_c(z)=\#\{K:z\mbox{ is one of the
  $r$ points of the $r$-tuple }\vec s_c^K\mbox{ in }S_c\}.$$
\end{defn}
According to the color scheme indexed by $c$, every point in the
configuration has a black degree, a red degree, etc.  Definition
\ref{def4.1} is stated in a way so that possibly repeated $r$-tuples
are counted with multiplicity.  The assumption that each $r$-tuple is
independent implies that $z$ appears at most once in an $r$-tuple.
The number $deg_c(z)$ does not depend on the choice of ordered
representative $S_c$ for $\SSS_c$, nor on ${\mathbf p}$ and $\ell$ if
$\SSS$ admits another description as a $({\mathbf
p^\prime},r,\ell^\prime,D)$-configuration.  For all but finitely many
points in $\fd P^D$, the $c$-degree is zero.

The following Definition is dual to Definition \ref{def4.1}.
\begin{defn}\label{def4.6}
  For $\SSS$ and $S_c$ as in Definition \ref{def4.1}, define the
  {\underline{$c$-degree}} of a projective $(r-1)$-subspace $L$ of
  $\fd P^D$ to be
  $$deg_c(L)=\#\{K:\mbox{all $r$ points of the $r$-tuple }\vec
  s_c^K\mbox{ in }S_c\mbox{ lie on }L\}.$$
\end{defn}

The following Definition of a morphism of a configuration was
motivated by, but is different from, a notion of isomorphic plane
configurations considered by \cite{shephard}.
\begin{defn}\label{def4.7}
  Given a $({\mathbf p},r,\ell,D)$-configuration
  $\SSS=(\SSS_0,\ldots,\SSS_n)$, and a $({\mathbf
  p},r,\ell,D^\prime)$-configuration $\ttt=(\ttt_0,\ldots,\ttt_n)$,
  $\aaa$ is a {\underline{morphism}} from $\SSS$ to $\ttt$ means
  $\aaa$ is a function $\pp(\SSS)\to\pp(\ttt)$ such that:
  \begin{enumerate}
    \item For indexing purposes, for any ordered
    representative for each $\SSS_c$, $c=0,\ldots,n$, 
    \begin{equation}\label{eq34}
      (\vec s_c^1,\ldots,\vec s_c^K,\ldots,\vec s_c^{\ \ell p_c}),
    \end{equation}
    there is an ordered representative for $\ttt_c$,
    \begin{equation}\label{eq35}
      (\vec t_c^{\ 1},\ldots,\vec t_c^{\ K},\ldots,\vec t_c^{\ \ell
    p_c}),
    \end{equation}
    and;
    \item There exists a function ${\mathbf
    A}:\uuu(\SSS)\to\fd^{D^\prime+1}$ such that the restriction of
    $\mathbf A$ to each of the subspaces ${\mathbf L}=\pi^{-1}(L)$ for
    $L\in\LL(\SSS)$ is one-to-one and $\fd$-linear, and induces a map
    $A_L:L\to\fd P^{D^\prime}$ which satisfies, for every $\vec s_c^K$
    that spans $L$:
    \begin{eqnarray}
      A_L(\vec
        s_c^K)&=&A_L\left(\left(s_c^{K,1},\ldots,s_c^{K,e},\ldots,s_c^{K,r}\right)\right)\nonumber\\
        &=&\left(A_L(s_c^{K,1}),\ldots,A_L(s_c^{K,e}),\ldots,A_L(s_c^{K,r})\right)\nonumber\\
        &=&\left(\aaa(s_c^{K,1}),\ldots,\aaa(s_c^{K,e}),\ldots,\aaa(s_c^{K,r})\right)\label{eq36}\\
        &=&\left(t_c^{K,1},\ldots,t_c^{K,e},\ldots,t_c^{K,r}\right)=\vec
        t_c^{\ K}.\nonumber
    \end{eqnarray}
  \end{enumerate}
\end{defn}
As a consequence of the Definition, a morphism defines a one-to-one
correspondence between the lists (\ref{eq34}) and (\ref{eq35}) of
$r$-tuples of color $c$, $c=0,\ldots,n$.  A morphism $\aaa$ is
necessarily an onto map on the sets of points,
$\pp(\SSS)\to\pp(\ttt)$, but is not necessarily one-to-one, and the
number $\#\LL(\ttt)$ may also be less than $\#\LL(\SSS)$.  Our notion
of morphism is a little stronger than just an incidence-preserving
collection of projective linear mappings $A_L$ of the projective
subspaces in $\LL(\SSS)$; the maps must all be induced by the same
$\mathbf A$.

\begin{center}
  \begin{figure}
\begin{center}
    \includegraphics[scale=0.5]{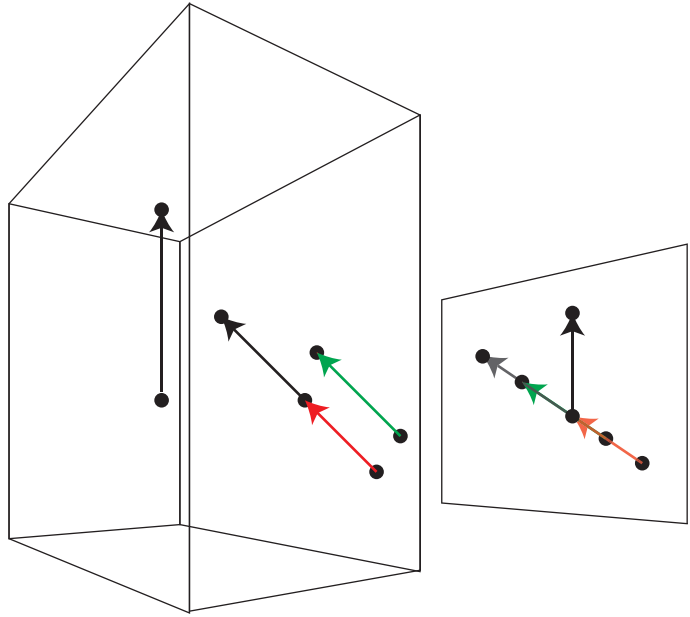}
    \caption{The projection from left to right defines a morphism from
      a configuration of $7$ points, $3$ lines, and $4$ segments in
      three dimensions to a configuration of $6$ points, $2$ lines,
      and $4$ segments in two dimensions.}\label{fig2}
  \end{center}
  \end{figure}
\end{center}

\begin{prop}\label{prop4.9}
  Given $\fd$, $\mathbf p$, $r$, $\ell$, the $D=r-1,\ldots,\infty$
  union of the sets of $({\mathbf p},r,\ell,D)$-configurations,
  together with the above notion of morphism, forms a category.
\end{prop}
\begin{proof}[Proof (sketch)]
  There is an identity morphism from any $\SSS$ to itself.  It is
  straightforward to check that the usual composition of maps of sets
  $\aaa:\pp(\SSS)\to\pp(\ttt)$ and ${\mathbf
  A}:\uuu(\SSS)\to\uuu(\ttt)$ defines an associative composition of
  morphisms.
\end{proof}
\begin{example}\label{ex4.8}
  The classical notion of projective equivalence is an important
  special case of morphism, as follows.  Let $D^\prime=D$, and let
  ${\mathbf A}$ be an invertible $\fd$-linear map
  $\fd^{D+1}\to\fd^{D+1}$.  The induced map $A:\fd P^D\to\fd P^D$ is a
  {\underline{projective transformation}} and a configuration $\SSS$
  is {\underline{projectively equivalent}} to its image $A(\SSS)$.
  The restriction of $A$ to $\pp(\SSS)$ is a morphism $\aaa$ from
  $\SSS$ to $A(\SSS)$ as in Definition \ref{def4.7}.  First, for any
  ordered representative of $\SSS_c$, $c=0,\ldots,n$, index the
  $r$-tuples in $A(\SSS_c)=\ttt_c$ by setting $\vec t_c^{\ K}=A(\vec
  s_c^K)$.  The map ${\mathbf A}:\uuu(\SSS)\to\fd^{D+1}$ from the
  Definition is just the restriction of the given linear map to
  $\uuu(\SSS)$, and restricts further to ${\mathbf L}=\pi^{-1}(L)$ for
  $L\in\LL(\SSS)$, so $\left.{\mathbf A}\right|_{\mathbf L}$ is
  one-to-one and linear, satisfying (\ref{eq1}) and (\ref{eq2}), so it
  induces $A_L:L\to\fd P^D$.  For each independent $r$-tuple $\vec
  s_c^{K}$ with span $L$, the induced map $A_L$ takes $\vec s_c^{K}$
  to an independent $r$-tuple $A_L(\vec s_c^K)=\vec t_c^{\ K}$.
\end{example}
\begin{example}\label{ex4.11}
  In Example \ref{ex4.8}, checking Definition \ref{def4.7} did not
  require that $D^\prime=D$, nor that $\mathbf A$ was invertible.  The
  same argument applies to any $\fd$-linear $\mathbf
  A:\fd^{D+1}\to\fd^{D^\prime+1}$, which is not necessarily one-to-one
  or onto, but which is one-to-one when restricted to subspaces
  ${\mathbf L}=\pi^{-1}(L)$ for $L\in\LL(\SSS)$.  As shown in Figure
  \ref{fig2}, the induced map $A$ could be a projection from a subset
  of a higher-dimensional projective space to a lower-dimensional
  space, and would define a morphism $\aaa$ from a configuration
  $\SSS$ to $A(\SSS)$ as long as the image of every
  $(r-1)$-dimensional projective subspace in $\LL(\SSS)$ is still
  $(r-1)$-dimensional.
\end{example}
\begin{defn}\label{def4.2}
  A $({\mathbf p},r,\ell,D)$-configuration $\SSS$ is a
  {\underline{weight $\mathbf p$ h-configuration}} means:
  \begin{enumerate}
    \item At every point $z\in\fd P^D$, these numbers are integers and
    are equal to each other:
      \begin{equation}\label{eq16}
        \frac{deg_0(z)}{p_0}=\cdots=\frac{deg_c(z)}{p_c}=\cdots=\frac{deg_n(z)}{p_n};
      \end{equation}    
    \item For every projective $(r-1)$-subspace $L\subseteq\fd P^D$,
    these numbers are integers and are equal to each other:
      \begin{equation}\label{eq39}
        \frac{deg_0(L)}{p_0}=\cdots=\frac{deg_c(L)}{p_c}=\cdots=\frac{deg_n(L)}{p_n}.
      \end{equation}
  \end{enumerate}
\end{defn}
For a weight $\mathbf p$ h-configuration $\SSS$, we have the following
geometric interpretation of the parameter $\ell$: if a
$(r-1)$-dimensional projective subspace $L$ in $\LL(\SSS)$ has
$deg_c(L)=m_Lp_c$, then by (\ref{eq39}), $m_L$ does not depend on $c$.
There is an unordered $\ell$-tuple of subspaces,
$[L_1,\ldots,L_k,\ldots,L_\ell]$, where each $L_k$ is incident with
exactly $p_c$ $r$-tuples with color $c$, and $L_k$ occurs in the
unordered list with multiplicity $m_{L_k}$.

\begin{lem}\label{lem4.13}
  If $\SSS$ is a weight $\mathbf p$ h-configuration and
  $\aaa:\SSS\to\ttt$ is a morphism, then $\ttt$ is a weight $\mathbf
  p$ h-configuration.
\end{lem}
\begin{proof}
  Let the ordered $\ell p_c$-tuple $S_c$ be an ordered representative
  for $\SSS_c$; then let $T_c$ be the corresponding ordered
  representative of $\ttt_c$ as in (\ref{eq35}).  The $r$ points in
  $\vec t_c^{\ K}$ are indexed, using (\ref{eq36}),
  \begin{equation}\label{eq38}
    t_c^{K,e}=\aaa(s_c^{K,e}),
  \end{equation}
  for $K=1,\ldots,\ell p_c$ and $e=1,\ldots,r$.  

  To check part 1.\ of Definition \ref{def4.2}, suppose $z\in\fd
  P^{D^\prime}$.  If $z\notin\pp(\ttt)$, then $deg_c(z)=0$ for all
  $c$.  If $z\in\pp(\ttt)$, then $\aaa^{-1}(z)$ is a finite set of
  points in $\pp(\SSS)$.  There is no $r$-tuple $\vec s_c^K$ that
  contains more than one point of $\aaa^{-1}(z)$, since $\aaa(\vec
  s_c^K)$ is the independent $r$-tuple $\vec t_c^{\ K}$.  An $r$-tuple
  $\vec t_c^{\ K}$ has $z$ as one of its $r$ points if and only if the
  corresponding $r$-tuple $\vec s_c^K$ has some element of
  $\aaa^{-1}(z)$ as one of its $r$ points.  For each $c$, the
  cardinality of the disjoint union of indices $K$ is:
  $$deg_c(z)=\sum_{w\in\aaa^{-1}(z)}deg_c(w).$$ The equalities in
  (\ref{eq16}) for $z$ follow from the assumed equalities for all the
  points $w$.

  Dually, projective $(r-1)$-subspaces not in $\LL(\ttt)$ have
  $deg_c=0$ for all $c$.  By (\ref{eq36}), every projective
  $(r-1)$-subspace in $\LL(\ttt)$ is of the form
  $A_{L^\prime}(L^\prime)$, and if $L^\prime$ is the span of $\vec
  s_c^{K^\prime}$, then all $r$ points $t_c^{K^\prime,e}$ lie on
  $A_{L^\prime}(L^\prime)$.  The set
  $\LL^\prime=\{L\in\LL(\SSS):A_L(L)=A_{L^\prime}(L^\prime)\}$ is
  finite, and there is no $r$-tuple $\vec s_c^K$ lying on more than
  one of these subspaces $L$.  An $r$-tuple $\vec t_c^{\ K}$ lies on
  $A_{L^\prime}(L^\prime)$ if and only if the corresponding $r$-tuple
  $\vec s_c^K$ lies on one of the subspaces $L\in \LL^\prime$.  For
  each $c$, the cardinality of the disjoint union of indices $K$ is:
  $$deg_c(A_{L^\prime}(L^\prime))=\sum_{L\in\LL^\prime}deg_c(L).$$
  The equalities in (\ref{eq39}) for
  $A_{L^\prime}(L^\prime)$ follow from the assumed
  equalities for all $L\in\LL^\prime$.
\end{proof}

\subsection{The Invariant}\label{sec4.2}

\begin{notation}\label{not4.3}
  Given an $(r-1)$-dimensional projective subspace $L$ in $\fd P^D$,
  which is the image of a $r$-dimensional subspace ${\mathbf L}$ in
  $\fd^{D+1}$, $L=\pi({\mathbf L})$, let $\bbb=({\mathbf
  b_0},\ldots,{\mathbf b_{r-1}})$ be an ordered basis for ${\mathbf
  L}$.  Given an independent set of vectors ${\mathbf s_e}\in {\mathbf
  L}$, $e=1,\ldots,r$, with $\pi({\mathbf s_e})=s_e$ on $L$, let $\vec
  s$ be the ordered $r$-tuple $(s_1,\ldots,s_r)$, and define
  $\llbracket\vec s\ \rrbracket_{\bbb}\in\fd^1_*$ by the following
  procedure.  The vectors have coordinates in the $\bbb$ basis:
  \begin{equation}\label{eq26}
    {\mathbf s_e}=s_e^0{\mathbf b_0}+\ldots+s_e^{r-1}{\mathbf
  b_{r-1}}\implies[{\mathbf
  s_e}]_\bbb=\left[\begin{array}{c}s_e^0\\\vdots\\s_e^{r-1}\end{array}\right]_\bbb\in\fd^r.
  \end{equation}
  By stacking columns into a square matrix, denote
   \begin{equation}\label{eq17}
    \llbracket\vec s\ \rrbracket_{\bbb}=\det\left(\left[[{\mathbf
    s_1}]_\bbb\cdots[{\mathbf s_e}]_\bbb\cdots[{\mathbf
    s_r}]_\bbb\right]_{r\times r}\right).
  \end{equation}
\end{notation} 
For example, in the $r=2$ case,
\begin{equation}\label{eq24}
    \llbracket\vec s\ \rrbracket_{\bbb}=s_2^1s_1^0-s_2^0s_1^1.
\end{equation}
The $r$-dimensional vector space ${\mathbf L}$, together with the
extra structure in the RHS of (\ref{eq17}), is called a
{\underline{Peano space}} by \cite{bbr}.  The {\underline{Peano
    bracket}} of $\vec s$ as we have defined it in (\ref{eq17})
depends on the choices of basis and representative points, and also on
the ordering of points in $\vec s$.  Note that picking a different
representative $\lambda\cdot{\mathbf s_e}$ for the point $s_e$ and
$\lambda\ne0$ transforms $\llbracket\vec s\ \rrbracket_{\bbb}$ to
$\lambda\cdot\llbracket\vec s\rrbracket_{\bbb}$.

\begin{thm}\label{thm4.4}
  Let $\SSS$ be a $({\mathbf p},r,\ell,D)$-configuration which is a
  weight $\mathbf p$ h-configuration.  For each $c=0,\ldots,n$, choose
  an ordered representative $S_c$ of $\SSS_c$ as in
  {\rm{(\ref{eq37})}}.  For each point $z$ in the set of points
  $\pp(\SSS)=\{s_c^{K,e}\}$, choose one representative vector
  ${\mathbf z}={\mathbf s}_c^{K,e}$.  For each projective
  $(r-1)$-subspace $L$ in the set $\LL(\SSS)$, choose one ordered
  basis $\bbb_L$ for the $r$-subspace ${\mathbf L}=\pi^{-1}(L)$, and
  if the span of $\vec s_c^K$ is $L$, denote $\bbb_{c,K}=\bbb_L$.
  Then the following element of $\fd P({\mathbf p})$ is well-defined,
  depending only on $\SSS$ and $\mathbf p$, and not the above choices.
  $$E_{\mathbf p}(\SSS)=\left[\prod_{K=1}^{\ell p_0}\llbracket\vec
    s_0^{K}\rrbracket_{\bbb_{0,K}}:\ldots:\prod_{K=1}^{\ell
    p_c}\llbracket\vec
    s_c^K\rrbracket_{\bbb_{c,K}}:\ldots:\prod_{K=1}^{\ell
    p_n}\llbracket\vec s_n^K\rrbracket_{\bbb_{n,K}}\right]_{\mathbf
    p}.$$ Further, if $\ttt$ is a $({\mathbf
    p},r,\ell,D^\prime)$-configuration and $\aaa:\SSS\to\ttt$ is a
    morphism, then $E_{\mathbf p}(\SSS)=E_{\mathbf p}(\ttt)$.
\end{thm}
\begin{proof}
  The choice of ordering $S_c$ as in (\ref{eq37}) is used only for
  well-defined indexing; the first thing to prove is that the
  $E_{\mathbf p}$ expression does not depend on this choice.  The
  second part of the Proof is to show the expression does not depend
  on the choices made in computing the bracket (\ref{eq17}).  The
  third part of the Proof is verifying the invariance under morphism.

  First, for each ordered $r$-tuple $\vec s_c^K$ in $S_c$, formula
  (\ref{eq17}) shows that the quantity $\llbracket\vec
  s_c^K\rrbracket_{\bbb_{c,K}}\in\fd^1_*$ depends on a choice of basis
  $\bbb_{c,K}$ and a choice of representative vectors for the $r$
  points.  By the independence property, the $r$ points of each $\vec
  s_c^K$ span a unique projective $(r-1)$-subspace $L$, for which a
  unique basis $\bbb_L$ was chosen, by hypothesis.  So, the basis used
  to compute $\llbracket\vec s_c^K\rrbracket_{\bbb_{c,K}}$ depends
  only on the $r$ points of $\vec s_c^K$ in $\fd P^D$.  Each of the
  points $s_c^{K,e}$ has a representative in $\fd^{D+1}_*$ that does
  not depend on the color index $c$ or the assignment of $K$ index to
  the $r$-tuple $\vec s_c^K$.  The construction as stated in the
  hypothesis requires picking the same representative vector $\mathbf
  z$ when a point appears more than once in the $\SSS$ configuration,
  in $r$-tuples with different indices or colors: if
  $z=s_c^{K,e}=s_{c^\prime}^{K^\prime,e^\prime}$ then ${\mathbf
    z}={\mathbf s}_{c}^{K,e}={\mathbf
    s}_{c^\prime}^{K^\prime,e^\prime}$.  We can conclude that
  $\llbracket\vec s_c^K\rrbracket_{\bbb_{c,K}}$ is computed using
  representative vectors of the points and a basis, both depending
  only on the $r$-tuple of points and not on the index $K$ coming from
  $S_c$.  By commutativity, the product
  $\displaystyle{\prod_{K=1}^{\ell\cdot p_c}\llbracket\vec
    s_c^K\rrbracket_{\bbb_{c,K}}}$ does not depend on the choice of
  ordered representative $S_c$ for $\SSS_c$, nor on ${\mathbf p}$,
  since $\ell\cdot p_c$ is uniquely determined by $\SSS$.  The element
  $$\left[\prod_{K=1}^{\ell p_0}\llbracket\vec
    s_0^{K}\rrbracket_{\bbb_{0,K}}:\ldots:\prod_{K=1}^{\ell
    p_n}\llbracket\vec s_n^K\rrbracket_{\bbb_{n,K}}\right]_{\mathbf
    p}\in\fd P({\mathbf p})$$ may depend on ${\mathbf p}$, as in
    Theorem \ref{thm1.12}.  We can conclude so far that the above
    expression depends only on $\SSS$ and $\mathbf p$, not on any of
    the choices of $S_c$.

  The independence property implies the quantities $\llbracket\vec
  s_c^K\rrbracket_{\bbb_{c,K}}$ are all non-zero, so each of the $n+1$
  components in the $E_{\mathbf p}$ expression is non-zero:
  $E_{\mathbf p}(\SSS)\in D_{\mathbf p}$.

  For the second part of the Proof, as previously mentioned, for each
  point $z$ occurring with any multiplicity in the $\SSS$
  configuration, the construction of the Theorem requires choosing a
  fixed representative $\mathbf z$.  Changing the choice of
  representative for that point, $\lambda\cdot{\mathbf z}$ instead of
  $\mathbf z$, changes each $\llbracket\vec
  s_c^K\rrbracket_{\bbb_{c,K}}$ quantity to
  $\lambda\cdot\llbracket\vec s_c^K\rrbracket_{\bbb_{c,K}}$, as
  remarked after Notation \ref{not4.3}, for every $\vec s_c^K$ that
  has $z$ as one of its $r$ points (and only one, by independence).
  In each expression $\displaystyle{\prod_{K=1}^{\ell\cdot
  p_c}\llbracket\vec s_c^K\rrbracket_{\bbb_{c,K}}}$ (with color index
  $c$), there are $deg_c(z)$ (possibly repeated) $r$-tuples $\vec
  s_c^K$ with $z$ as one of its $r$ points, so changing ${\mathbf z}$
  to $\lambda\cdot{\mathbf z}$ changes the product expression by a
  factor of $\lambda^{deg_c(z)}$.  By part 1.\ of Definition
  \ref{def4.2}, there is some integer $y$ depending on $z$ but not
  $c$, so that $deg_c(z)=y\cdot p_c$.  Since for each $c$, the product
  changes by a factor of $(\lambda^y)^{p_c}$, the $\sim_{\mathbf
  p}$-equivalence class of the $E_{\mathbf p}$ expression does not
  depend on the choice of $\lambda$ or ${\mathbf z}$.

  For a projective $(r-1)$-subspace $L$, the value of the bracket
  $\llbracket\vec s_c^K\rrbracket_{\bbb_{c,K}}$ depends on the choice
  of ordered basis $\bbb_{c,K}=\bbb_L=\left({\mathbf
    b}_{0},\ldots,{\mathbf b}_{r-1}\right)$ in the following way: let
  $\bbb_L^\prime$ be another ordered basis of the same $r$-dimensional
  space ${\mathbf L}$.  Then there exists a $r\times r$ invertible
  matrix $Q$ which changes $\bbb_L$-coordinates to
  $\bbb_L^\prime$-coordinates, via matrix multiplication: if the
  $\bbb_L$-coordinate column vector of ${\mathbf s}_c^{K,e}$ is as
  in (\ref{eq26}), then the $\bbb_L^\prime$-coordinate column vector
  is $[{\mathbf s}_c^{K,e}]_{\bbb_L^\prime}=Q[{\mathbf
      s}_c^{K,e}]_{\bbb_L}$.  Applying the $Q$ coordinate change
  matrix to each column in the determinant (\ref{eq17}) transforms the
  bracket by the well-known formula
  \begin{eqnarray*}
    \llbracket\vec s_c^K\rrbracket_{\bbb_L^\prime}&=&\det\left(\left[(Q[{\mathbf
    s}_c^{K,1}]_{\bbb_L})\cdots(Q[{\mathbf
    s}_c^{K,r}]_{\bbb_L})\right]_{r\times r}\right)\\
    &=&\det(Q)\det\left(\left[[{\mathbf
    s}_c^{K,1}]_{\bbb_L}\cdots[{\mathbf
    s}_c^{K,r}]_{\bbb_L}\right]_{r\times r}\right)\\
    &=&\det(Q)\llbracket\vec s_c^{K}\rrbracket_{\bbb_L}.
  \end{eqnarray*}
  We can conclude that for any $L$, changing the choice of ordered
  basis $\bbb_L$ to a new basis $\bbb_L^\prime$, and using this new
  basis for every bracket expression for an $r$-tuple on $L$, results
  in changing each expression with color index $c$,
  $\displaystyle{\prod_{K=1}^{\ell p_c}\llbracket\vec
    s_c^K\rrbracket_{\bbb_{c,K}}}$, by a factor of
  $(\det(Q))^{deg_c(L)}$, where $deg_c(L)=m_Lp_c$, and $m_L$ does not
  depend on $c$, by part 2.\ of Definition \ref{def4.2}.  Since for
  each $c$, the product changes by a factor of
  $(\det(Q)^{m_L})^{p_c}$, the $\sim_{\mathbf p}$-equivalence class of
  $E_{\mathbf p}$ is unchanged.  This shows that $E_{\mathbf p}$ does
  not depend on the choices made as in the statement of the Theorem,
  which are required to compute the brackets $\llbracket\vec
  s_c^K\rrbracket_{\bbb_L}$.

  Thirdly, by Lemma \ref{lem4.13}, if there is a morphism
  $\aaa:\SSS\to\ttt$, then $\ttt$ is also a weight $\mathbf p$
  h-configuration, and so the expression $E_{\mathbf p}(\ttt)$ is
  well-defined by the previous part of this Proof.  As in the Proof of
  Lemma \ref{lem4.13}, an ordering for $\SSS_c$ corresponds to one for
  $\ttt_c$, giving an indexing as in (\ref{eq38}). 

  For each projective $(r-1)$-subspace $A_{L}(L)$ in the set
  $\LL(\ttt)$, pick an ordered basis $\cc$ for the linear $r$-subspace
  ${\mathbf A}|_{{\mathbf L}}({\mathbf L})$, as in the hypothesis of
  the Theorem applied to $\LL(\ttt)$.  For any ${\mathbf L}^\prime$
  with $A_{L^\prime}(L^\prime)=A_L(L)$, ${\mathbf A}|_{{\mathbf
  L}^\prime}$ is linear and one-to-one, so $\left({\mathbf
  A}|_{{\mathbf L}^\prime}\right)^{-1}(\cc)$ is an ordered basis for
  ${\mathbf L}^\prime$, and setting $\bbb_{L^\prime}=\left({\mathbf
  A}|_{{\mathbf L}^\prime}\right)^{-1}(\cc)$ satisfies the uniqueness
  hypothesis of the Theorem applied to $\LL(\SSS)$.

  Dually, for each point $w=t_c^{K,e}$ in the set $\pp(\ttt)$, pick a
  representative ${\mathbf w}={\mathbf t}_c^{K,e}$ in
  $\fd^{D^\prime+1}_*$ as in the hypothesis.  For an index $(c,K,e)$,
  the point $s_c^{K,e}$ lies on a $(r-1)$-subspace $L_{c,K}$ spanned
  by $\vec s_c^K$, and satisfies $A_{L_{c,K}}(s_c^{K,e})=t_c^{K,e}$,
  and has a representative vector $({\mathbf A}|_{{\mathbf
      L_{c,K}}})^{-1}({\mathbf t}_c^{K,e})$ in $\fd^{D+1}_*$.  To show
  that this representative vector depends only on the point and not on
  the index, suppose $(c^\prime,K^\prime,e^\prime)$ is any other index
  with $s_c^{K,e}=s_{c^\prime}^{K^\prime,e^\prime}$; then the point is
  on both projective $(r-1)$-subspaces $L_{c,K}$ and
  $L_{c^\prime,K^\prime}$, and ${\mathbf A}|_{{\mathbf L}_{c,K}}$ and
  ${\mathbf A}|_{{\mathbf L}_{c^\prime,K^\prime}}$ agree on the
  intersection ${\mathbf L}_{c,K}\cap{\mathbf L}_{c^\prime,K^\prime}$
  because they are restrictions of the same map $\mathbf A$.  Since
  ${\mathbf t}_c^{K,e}={\mathbf t}_{c^\prime}^{K^\prime,e^\prime}$, we
  can conclude $({\mathbf A}|_{{\mathbf L}_{c,K}})^{-1}({\mathbf
    t}_c^{K,e})=({\mathbf A}|_{{\mathbf
      L}_{c^\prime,K^\prime}})^{-1}({\mathbf
    t}_{c^\prime}^{K^\prime,e^\prime})$, and denote this
  representative vector ${\mathbf s}_c^{K,e}$.

  Now, fix an index pair $(c,K)$ and consider corresponding $r$-tuples
  $\vec s_c^K$ and $\vec t_c^{\ K}$, lying on subspaces $L_{c,K}$ and
  $A_{L_{c,K}}(L_{c,K})$ as above.  The coordinate vector of ${\mathbf
  t}_c^{K,e}$ with respect to the ordered basis $\cc=\left({\mathbf
  c}_{0},\ldots,{\mathbf c}_{r-1}\right)$ of ${\mathbf A}|_{{\mathbf
  L_{c,K}}}({\mathbf L}_{c,K})$ is related to the coordinate vector of
  ${\mathbf s}_c^{K,e}$ with respect to the ordered basis
  $\bbb_{c,K}=\left({\mathbf A}|_{{\mathbf L_{c,K}}}\right)^{-1}(\cc)$
  of ${\mathbf L}_{c,K}$, by the linearity of ${\mathbf A}|_{{\mathbf
  L}_{c,K}}$:
  \begin{eqnarray*}
    ({\mathbf s}_c^{K,e})&=&\left({\mathbf A}|_{{\mathbf
  L}_{c,K}}\right)^{-1}({\mathbf t}_c^{K,e})\\ &=&\left({\mathbf
  A}|_{{\mathbf L}_{c,K}}\right)^{-1}(t_c^{K,e,0}{\mathbf
  c}_{0}+\ldots+t_c^{K,e,r-1}{\mathbf c}_{r-1})\\
  &=&t_c^{K,e,0}\left({\mathbf A}|_{{\mathbf
  L}_{c,K}}\right)^{-1}({\mathbf
  c}_{0})+\ldots+t_c^{K,e,r-1}\left({\mathbf A}|_{{\mathbf
  L}_{c,K}}\right)^{-1}({\mathbf c}_{r-1})\\
  &=&t_c^{K,e,0}{\mathbf b}_{c,K,0}+\ldots+t_c^{K,e,r-1}{\mathbf
  b}_{c,K,r-1},
  \end{eqnarray*}
  i.e., the $\bbb_{c,K}$-coordinates of ${\mathbf s}_c^{K,e}$ are the
  same as the $\cc$-coordinates of ${\mathbf t}_c^{K,e}$, and
  \begin{eqnarray*}
  \llbracket\vec
    s_c^K\rrbracket_{\bbb_{c,K}}&=&\det\left(\left[[{\mathbf
    s}_c^{K,1}]_{\bbb_{c,K}}\cdots[{\mathbf
    s}_c^{K,r}]_{\bbb_{c,K}}\right]_{r\times r}\right)\\
    &=&\det\left(\left[[{\mathbf t}_c^{K,1}]_{\cc}\cdots[{\mathbf
    t}_c^{K,r}]_{\cc}\right]_{r\times r}\right)\\ &=&\llbracket\vec
    t_c^{\ K}\rrbracket_{\cc}.
  \end{eqnarray*}
  Using these brackets to compute the products in the $E_{\mathbf
  p}(\SSS)$ expression, and the previously established fact that
  $E_{\mathbf p}(\SSS)$ does not depend on the indexing $S_c$, or the
  choices of $\bbb_L$ or representative vectors, the claimed equality
  $E_{\mathbf p}(\SSS)=E_{\mathbf p}(\ttt)$ is proved.
\end{proof}

\begin{example}\label{ex4.5}
  For ${\mathbf p}=(1,1)$, $n=1$ and there are two colors.
  $E_{(1,1)}(\SSS)$ is a ratio of products of $\ell$ determinants of
  size $r\times r$, which, as stated in the Introduction, would have
  been recognizable before Eves' time.  The case ${\mathbf p}=(1,1)$,
  $r=2$, of Theorem \ref{thm4.4} can be called a purely projective, or
  algebraic, version of Eves' Theorem, in comparison to the Euclidean,
  or metric, version, Theorem 6.2.2 of \cite{e}.  The connection
  between the determinantal expression and Eves' formula involving
  Euclidean signed lengths in $\re^D$ is discussed in Section
  \ref{sec4}.
  
  For $r=2$, in a $((1,1),2,\ell,D)$-configuration $\SSS$, $\SSS_0$ is
  a list of $\ell$ black directed segments in $\fd P^D$, and $\SSS_1$
  is a list of $\ell$ red segments.  If $\SSS$ is a weight $\mathbf p$
  h-configuration (which in this ${\mathbf p}=(1,1)$, $r=2$ case we
  just call an {\underline{h-configuration}}), then there are $\ell$
  (counting with multiplicity) lines $[L_1,\ldots,L_\ell]$ with one
  black segment $\vec s_0^K$ and one red segment $\vec s_1^K$ on each
  line, and at each point, the black degree equals the red degree.
  The following element of $\fd P^1$, where each expression
  $\llbracket\vec s_c^K\rrbracket_{\bbb_{c,K}}$ is calculated as in
  Theorem \ref{thm4.4}, is well-defined and invariant under projective
  transformations of $\fd P^D$:
  $$E_{(1,1)}(\SSS)=\left[\prod_{K=1}^\ell\llbracket\vec
  s_0^K\rrbracket_{\bbb_{0,K}}:\prod_{K=1}^\ell\llbracket\vec
  s_1^K\rrbracket_{\bbb_{1,K}}\right].$$ Eves calls the ratio
  $$\frac{\displaystyle{\prod_{K=1}^\ell\llbracket\vec
  s_1^K\rrbracket_{\bbb_{1,K}}}}{\displaystyle{\prod_{K=1}^\ell\llbracket\vec
  s_0^K\rrbracket_{\bbb_{0,K}}}}$$ an ``h-expression'': each line
  $L_K$ occurs equally often (multiplicity $m_{L_K}$) in the numerator
  and denominator, and each point in $\pp(\SSS)$ occurs equally often
  in the numerator (red degree) and denominator (black degree).
\end{example}
\begin{example}\label{ex4.6}
  Consider four distinct points $\alpha$, $\beta$, $\gamma$, $\delta$
  on the projective line $\fd P^1$.  These can be organized into an
  h-configuration $\SSS$, with ${\mathbf p}=(1,1)$ and $r=2$ as in
  Example \ref{ex4.5}, dimension $D=1$, and $\ell=2$.  Let
  $\SSS_0=[(\delta,\alpha),(\gamma,\beta)]$ be a list of black
  segments, and let $\SSS_1=[(\gamma,\alpha),(\delta,\beta)]$ be a
  list of red segments, as shown in Figure \ref{fig3}.  Then
  $\LL(\SSS)$ is the singleton set $\{L=\fd P^1\}$; we could, as
  mentioned after Definition \ref{def4.2}, consider the line occurring
  with multiplicity two in the unordered list $[L_1,L_2]$ with
  $L=L_1=L_2$.  Choose the standard ordered basis
  $\bbb_L=\left((1,0),(0,1)\right)$ of $\fd^2$, so $\alpha$ has
  homogeneous coordinates $[\alpha_0:\alpha_1]$, vector representative
  $\alpha_0{\mathbf b}_{0}+\alpha_1{\mathbf b}_{1}$, and
  $\bbb_L$-coordinate vector
  $\left[\begin{array}{c}\alpha_0\\\alpha_1\end{array}\right]$, and
  similarly for the other points.  Let $S_0=(\vec
  s_0^1=(\delta,\alpha),\vec s_0^{\ 2}=(\gamma,\beta))$ be an ordered
  representative of $\SSS_0$ and let $S_1=(\vec
  s_1^1=(\gamma,\alpha),\vec s_1^{\ 2}=(\delta,\beta))$ be an ordered
  representative of $\SSS_1$.  Each endpoint has black degree and red
  degree both equal to $1$, and the line $L$ satisfies Part 2.\ of
  Definition \ref{def4.2}, with $deg_0(L)=deg_1(L)=2$.  Alternatively,
  we could assign one of the black segments and one of the red
  segments to $L_1=L$, and the remaining segments to $L_2=L$; there
  are various choices of such assignments, which would not affect the
  expression (\ref{eq31}).  The expression from Theorem \ref{thm4.4}
  is:
  \begin{eqnarray}
    E_{(1,1)}(\SSS)&=&\left[\llbracket\vec
    s_0^1\rrbracket_{\bbb_L}\llbracket\vec s_0^{\
    2}\rrbracket_{\bbb_L}:\llbracket\vec
    s_1^1\rrbracket_{\bbb_L}\llbracket\vec s_1^{\
    2}\rrbracket_{\bbb_L}\right]\label{eq31}\\
    &=&\left[(\alpha_1\delta_0-\alpha_0\delta_1)(\beta_1\gamma_0-\beta_0\gamma_1):(\alpha_1\gamma_0-\alpha_0\gamma_1)(\beta_1\delta_0-\beta_0\delta_1)\right],\nonumber
  \end{eqnarray}
  which is exactly the well-known cross-ratio of the ordered quadruple
  $(\alpha,\beta,\gamma,\delta)$.

\begin{center}
  \begin{figure}
\begin{center}
    \includegraphics[scale=0.6]{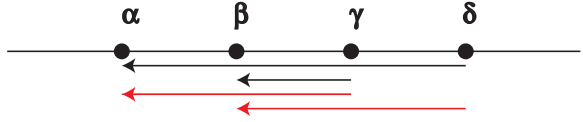}
    \caption{A configuration of $4$ points, $1$ line, and $4$ ordered
      pairs, as indicated by the red and black arrows drawn offset
      from the line.}\label{fig3}
  \end{center}
  \end{figure}
\end{center}

\end{example}
In classical Invariant Theory, the fundamental property of projective
invariance of the cross-ratio was often proved using determinants and
algebraic methods similar to our Proof of Theorem \ref{thm4.4} (e.g.,
\cite{c}; \cite{s} Arts.\ XIII.136, 137, XVII.195).  In projective
geometry, the general idea that projective transformations introduce
canceling factors in certain product expressions already appears in
(\cite{p} \S 20).

\section{Metric versions}\label{sec4}

Eves' Theorem as stated in \cite{e} is about ratios of signed lengths
of directed segments, in the real Euclidean plane extended to include
points at infinity.  The earlier identities of \cite{p}, and
interesting applications of Eves' Theorem, including Ceva's Theorem
and others appearing in (\cite{e} \S 6.2), \cite{f}, and
\cite{shephard}, also involve Euclidean distance between pairs of
points.  The constructions in Section \ref{sec3} were developed in
terms of linear algebra and projective geometry, avoiding any notion
of distance.  However, there are connections between projective
geometry and Euclidean geometry --- a thorough, modern treatment is
given by \cite{rg}, relating Cartesian coordinates in affine
neighborhoods and bracket operations (as in the above Notation
\ref{not4.3}) to distance, area, volume, angles, etc.

For this Section, we consider only the case $\fd=\re$, and start by
incorporating a notion of distance, as a bit of extra structure added
to the projective coordinate system.

Consider, as in Example \ref{ex1.2}, $\re^{D+1}$ with coordinates
${\mathbf x}=(x_0,x_1,\ldots,x_D)$, the projection
$\pi:\re^{D+1}_*\to\re P^D$, and homogeneous coordinates
$x=[x_0:x_1:\ldots:x_D]$ for $\re P^D$.  The restriction of $\pi$ to
the hyperplane $\{(1,x_1,x_2,\ldots,x_D)\}$ is one-to-one onto the
image $\{x:x_0\ne0\}$ in $\re P^D$.  We can refer to this affine
neighborhood as $\re^D$, where a point in $\re^D$ has both homogeneous
and affine coordinates: $x=[1:x_1:\ldots:x_D]=(x_1,\ldots,x_D)$, and
also is the image of a representative vector: $x=\pi({\mathbf
  x})=\pi(1,x_1,\ldots,x_D)$.

The extra structure we choose to assign to the affine neighborhood
$\re^D$ is that of a {\underline{normed vector space}}, where the
vector space structure is the usual one from the affine coordinate
system $(x_1,\ldots,x_D)$, and $\|\ . \ \|$ is any norm function.
Then there is a {\underline{distance function}} on $\re^D$: ${\mathbf
  d}(x,y)=\|y-x\|$.

In the $r=2$ case, we are interested in directed segments on lines.
Given a line $L$ in $\re^D$ (meaning, a non-empty intersection of a
projective line $L=\pi({\mathbf L})$ with the $\{x:x_0\ne0\}$
neighborhood), it can be parametrized by choosing a start point $b_0$
and a non-zero direction vector $v$, so $L=\{b_0+tv:t\in\re\}$.  The
choice of $v$ also determines a {\underline{direction}} for the line:
an ordered pair of distinct points $(b_0+t_1v,b_0+t_2v)$ is a
positively (or negatively) directed segment depending on the sign of
$t_2-t_1$.  There exists a unique $t$ value so that $t>0$ and the
point $b_1=b_0+tv$ satisfies ${\mathbf d}(b_0,b_1)=1$.  Choose these
representative vectors in $\re^{D+1}$ for $b_0$ and $b_1$: ${\mathbf
  b_0}=(1,b_0^1,\ldots,b_0^D)$ and ${\mathbf
  b_1}=(1,b_1^1,\ldots,b_1^D)$.  So, choosing a start point and a
direction for the affine line $L$ determines (and is determined by) an
ordered basis $\bbb=({\mathbf b_0},{\mathbf b_1})$ (with both points
in $\{x_0=1\}$) for the plane ${\mathbf L}$.

Consider two distinct points $\alpha$, $\beta$ on the line $L$ in
$\re^D\subseteq\re P^D$.  If we re-parametrize $L$ using $b_1-b_0$ as
a direction vector,
\begin{eqnarray*}
  \alpha&=&(b_0^1+t_1(b_1^1-b_0^1),\ldots,b_0^D+t_1(b_1^D-b_0^D)),\\
  \beta&=&(b_0^1+t_2(b_1^1-b_0^1),\ldots,b_0^D+t_2(b_1^D-b_0^D)).
\end{eqnarray*}
The distance from $\alpha$ to $\beta$ does not depend on the choice of
start point $b_0$ nor the direction; it satisfies:
\begin{eqnarray*}
  {\mathbf d}(\alpha,\beta)&=&\|\beta-\alpha\|=\|(t_2-t_1)(b_1-b_0)\|\\
  &=&|t_2-t_1|\|b_1-b_0\|=|t_2-t_1|.
\end{eqnarray*}
The {\underline{signed length}} of the directed segment
$(\alpha,\beta)=\overrightarrow{\alpha\beta}$ is $t_2-t_1$, which
depends on the direction but not the start point.  Choosing the
representative vectors
\begin{eqnarray}
  \ba&=&(1,b_0^1+t_1(b_1^1-b_0^1),\ldots,b_0^D+t_1(b_1^D-b_0^D))=(1-t_1){\mathbf b_0}+t_1{\mathbf b_1},\label{eq18}\\
  \bbeta&=&(1,b_0^1+t_2(b_1^1-b_0^1),\ldots,b_0^D+t_2(b_1^D-b_0^D))=(1-t_2){\mathbf b_0}+t_2{\mathbf b_1},\nonumber
\end{eqnarray}
the signed length is exactly the same as the bracket formula
(\ref{eq24}):
\begin{eqnarray*}
  \llbracket\overrightarrow{\alpha\beta}\rrbracket_\bbb&=&t_2(1-t_1)-(1-t_2)t_1=t_2-t_1.
\end{eqnarray*}
When $\llbracket\overrightarrow{\alpha\beta}\rrbracket_\bbb$
expressions are used in Theorem \ref{thm4.4}, the expression
$E_{\mathbf p}(\SSS)$ does not depend on the choice of representative
vectors $\ba$, $\bbeta$ as long as representatives are chosen
consistently (as in (\ref{eq18})), nor on the choice of $\bbb$ as long
as that ordered basis is used for all directed segments on that line.
Since the construction requires each line $L$ to be assigned a unique
ordered basis $\bbb_L$, each line can have its own choice of direction
determined by $\bbb_L$, and a unit of length depending on a norm $\|\
. \ \|_L$.  So, a metric version for the $r=2$ (directed segments)
case of Theorem \ref{thm4.4} can be stated as follows.
\begin{cor}\label{cor5.1}
  Given a weight $\mathbf p$ h-configuration $\SSS$ of points in
  $\re^D$ with $r=2$, choose ordered representatives as in
  \mbox{\rm{(\ref{eq37})}}, and for each line $L$ in the set
  $\LL(\SSS)$, choose a direction and unit of length so that
  $\llbracket\overrightarrow{\alpha\beta}\rrbracket$ denotes the
  signed length of directed segments on the line through $\alpha$ and
  $\beta$.  The following element of $\re P({\mathbf p})$ does not
  depend on the choices of ordered representatives, directions, or
  unit lengths.
  $$E_{\mathbf p}(\SSS)=\left[\prod_{K=1}^{\ell p_0}\llbracket\vec
    s_0^K\rrbracket:\ldots:\prod_{K=1}^{\ell p_c}\llbracket\vec
    s_c^K\rrbracket:\ldots:\prod_{K=1}^{\ell p_n}\llbracket\vec
    s_n^K\rrbracket\right]_{\mathbf p}.$$ Further, $E_{\mathbf
    p}(\SSS)$ is invariant under a morphism that maps the points in
    $\SSS$ into an affine neighborhood $\re^{D^\prime}$.  \boxx
\end{cor}

The Peano bracket also admits a Euclidean interpretation in the above
coordinate system for configurations with $r=3$ and $D=2$ (see
\cite{bb}, \cite{crg}, in addition to the previously mentioned
\cite{rg}).  However, in order for the bracket to define a Euclidean
area in $\re^2$, we must use the Euclidean magnitude $\|\ . \ \|$,
defined by the standard dot product in the affine coordinate system
$(x_1,x_2)$.  Let $L$ be the entire real projective plane $L=\re P^2$,
and let ${\mathbf L}=\re^3$.  Pick the standard ordered basis $\bbb$,
so that three points $\alpha=(\alpha_1,\alpha_2)$,
$\beta=(\beta_1,\beta_2)$, $\gamma=(\gamma_1,\gamma_2)$ in
$\re^2\subseteq\re P^2$ have representatives with $\bbb$-coordinates
$\ba=\left[\begin{array}{c}1\\\alpha_1\\\alpha_2\end{array}\right]_\bbb$,
etc.  Then,
$$\llbracket(\alpha,\beta,\gamma)\rrbracket_\bbb=\det\left(\left[\left[\begin{array}{c}1\\\alpha_1\\\alpha_2\end{array}\right]_\bbb\left[\begin{array}{c}1\\\beta_1\\\beta_2\end{array}\right]_\bbb\left[\begin{array}{c}1\\\gamma_1\\\gamma_2\end{array}\right]_\bbb\right]_{3\times3}\right)=2Area\triangle(\alpha\beta\gamma),$$
twice the signed area of the triangle (\cite{e} \S 2.1), which depends
on the ordering of the three points and the (previously chosen)
standard Euclidean structure on $\re^2$.  The points with affine
coordinates $(0,0)$, $(1,0)$, $(0,1)$, in that order, form a
counter-clockwise triangle with positive area $\frac12$.

The following Example of ratios of areas was described by \cite{c} as
a ``graphometric'' quantity: a Euclidean measurement invariant under
projective transformations.
\begin{example}\label{ex5.2}
  Consider six points, labeled ${\mathnormal1},{\mathnormal2},{\mathnormal3},{\mathnormal4},{\mathnormal5},{\mathnormal6}$, in the plane
  $\re^2\subseteq\re P^2$.  They can be organized into a
  $((1,1),3,2,2)$-configuration $\SSS=(\SSS_0,\SSS_1)$, where
  $$\SSS_0=[\triangle({\mathnormal1}{\mathnormal2}{\mathnormal4}),\triangle({\mathnormal3}{\mathnormal5}{\mathnormal6})]$$
  is a list of two black triangles, and
  $$\SSS_1=[\triangle({\mathnormal1}{\mathnormal2}{\mathnormal3}),\triangle({\mathnormal4}{\mathnormal5}{\mathnormal6})]$$
  is a list of two red triangles (assuming non-collinearity of the
  indicated triples), as in Figure \ref{fig4}.  Then
  $\LL(\SSS)=\{L=\re P^2\}$, and we choose the standard basis $\bbb$
  as above.  As in Example \ref{ex4.6}, $deg_0(L)=deg_1(L)=2$, or we
  could consider a list $[L_1,L_2]$ with $L_1=L_2=L$, and one black
  triangle and one red triangle is assigned to each of $L_1$ and
  $L_2$.  Each of the six points in $\pp(\SSS)$ is a vertex of one
  black triangle and one red triangle, so the black degree equals the
  red degree and $\SSS$ is a weight $(1,1)$ h-configuration.  The
  invariant from Theorem \ref{thm4.4} is analogous to (\ref{eq31}):
  \begin{eqnarray*}
    E_{(1,1)}(\SSS)&=&\left[\llbracket\vec s_0^1\rrbracket_{\bbb}\llbracket\vec s_0^{\ 2}\rrbracket_{\bbb}:\llbracket\vec s_1^1\rrbracket_{\bbb}\llbracket\vec s_1^{\ 2}\rrbracket_{\bbb}\right]\\
    &=&\left[\llbracket\triangle({\mathnormal1}{\mathnormal2}{\mathnormal4})\rrbracket_{\bbb}\llbracket\triangle({\mathnormal3}{\mathnormal5}{\mathnormal6})\rrbracket_{\bbb}:\llbracket\triangle({\mathnormal1}{\mathnormal2}{\mathnormal3})\rrbracket_{\bbb}\llbracket\triangle({\mathnormal4}{\mathnormal5}{\mathnormal6})\rrbracket_{\bbb}\right].
  \end{eqnarray*}
  We can conclude that the ratio of signed areas
  $$\frac{(Area\triangle({\mathnormal1}{\mathnormal2}{\mathnormal3}))(Area\triangle({\mathnormal4}{\mathnormal5}{\mathnormal6}))}{(Area\triangle({\mathnormal1}{\mathnormal2}{\mathnormal4}))(Area\triangle({\mathnormal3}{\mathnormal5}{\mathnormal6}))}$$
  is an invariant of the configuration $\SSS$ under projective
  transformations (that do not send any of the six points to
  infinity).
\end{example}
  We remark that the property $E_{(1,1)}(\SSS^{(0,1)})=[1:1]$, or
  equivalently
  $\llbracket\triangle({\mathnormal1}{\mathnormal2}{\mathnormal4})\rrbracket_{\bbb}\llbracket\triangle({\mathnormal3}{\mathnormal5}{\mathnormal6})\rrbracket_{\bbb}-\llbracket\triangle({\mathnormal1}{\mathnormal2}{\mathnormal3})\rrbracket_{\bbb}\llbracket\triangle({\mathnormal4}{\mathnormal5}{\mathnormal6})\rrbracket_{\bbb}=0$,
  admits a projective (not necessarily Euclidean) interpretation as
  the concurrence of the lines through
  $\{{\mathnormal1},{\mathnormal2}\}$,
  $\{{\mathnormal3},{\mathnormal4}\}$,
  $\{{\mathnormal5},{\mathnormal6}\}$ (\cite{crg}, \cite{rg} Ch.\ 6).

\begin{center}
  \begin{figure}
\begin{center}
    \includegraphics[scale=0.6]{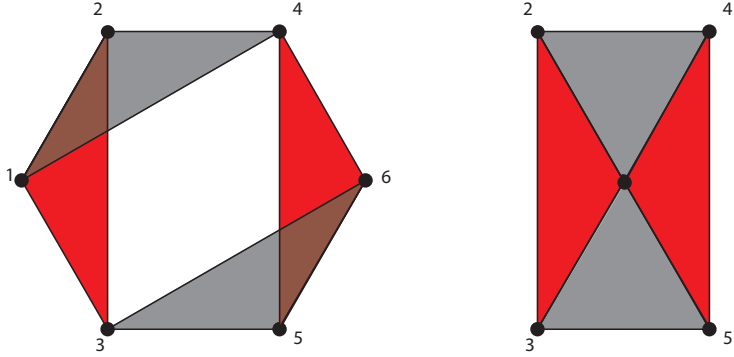}
    \caption{Left: A configuration of $6$ points and $4$ triangles in
      the real plane, from Example \ref{ex5.2}.  Right: The points
      ${\mathnormal1}$ and ${\mathnormal6}$ coincide, as in Example
      \ref{ex5.3}.}\label{fig4}
  \end{center}
  \end{figure}
\end{center}

\begin{example}\label{ex5.3}
  In the configuration from Example \ref{ex5.2}, it is possible for
  $\SSS$ to be a weight $(1,1)$ h-configuration even if points
  ${\mathnormal1}$ and ${\mathnormal6}$ coincide, as in Figure
  \ref{fig4}.  This gives another well-known $E_{(1,1)}$ projective
  invariant for five points in the (projective or Euclidean) plane
  (\cite{rg} \S 10.2),
  $$\frac{\llbracket\triangle({\mathnormal1}{\mathnormal2}{\mathnormal3})\rrbracket_{\bbb}\llbracket\triangle({\mathnormal4}{\mathnormal5}{\mathnormal1})\rrbracket_{\bbb}}{\llbracket\triangle({\mathnormal1}{\mathnormal2}{\mathnormal4})\rrbracket_{\bbb}\llbracket\triangle({\mathnormal3}{\mathnormal5}{\mathnormal1})\rrbracket_{\bbb}}=\frac{(Area\triangle({\mathnormal1}{\mathnormal2}{\mathnormal3}))(Area\triangle({\mathnormal4}{\mathnormal5}{\mathnormal1}))}{(Area\triangle({\mathnormal1}{\mathnormal2}{\mathnormal4}))(Area\triangle({\mathnormal3}{\mathnormal5}{\mathnormal1}))}.$$
\end{example}

\section{Reconstruction}\label{sec5}

Let $\SSS=(\SSS_0,\ldots,\SSS_n)$ be a $({\mathbf
p},r,\ell,D)$-configuration in $\fd P^D$, and pick a pair of colors
$(i,j)\in I$.  The ordered pair $(\SSS_i,\SSS_j)$ is a
$((p_i,p_j),r,\ell,D)$-configuration.  If $\SSS$ is a weight $\mathbf
p$ h-configuration, then $(\SSS_i,\SSS_j)$ is a weight $(p_i,p_j)$
h-configuration.  For a weight $\mathbf p$ h-configuration $\SSS$, the
following are equivalent:
  \begin{enumerate}
    \item $(\SSS_i,\SSS_j)$ is a $((1,1),r,\ell\cdot
    p_i,D)$-configuration and a weight $(1,1)$ h-configuration;
    \item $p_i=p_j$.
\end{enumerate}
The goal of the following construction is to modify a
$((p_i,p_j),r,\ell,D)$-configuration $(\SSS_i,\SSS_j)$ into a new
$((1,1),r,\ell^\prime,D)$-configuration $\SSS^{(i,j)}$ in a way such
that if $p_i=p_j$, then the configuration does not change:
$\SSS^{(i,j)}=(\SSS_i,\SSS_j)$, and if $(\SSS_i,\SSS_j)$ is a weight
$(p_i,p_j)$ h-configuration, then $\SSS^{(i,j)}$ is a weight $(1,1)$
h-configuration.

Recall $\ell_{ij}=\lcm(p_i,p_j)$, and let $a_{ij}=\ell_{ij}/p_i$,
$b_{ij}=\ell_{ij}/p_j$ as in (\ref{eq20}).  
\begin{notation}\label{not6.1}
  Given a $((p_i,p_j),r,\ell,D)$-configuration $(\SSS_i,\SSS_j)$,
  define a new ordered pair
  $\SSS^{(i,j)}=(\SSS^{(i,j)}_i,\SSS^{(i,j)}_j)$, where as in
  (\ref{eq21}), each entry is an unordered list of $r$-tuples of
  points, one list with color $i$, the other with color $j$.  Let
  $\SSS^{(i,j)}_i$ be the concatenation of $a_{ij}$ copies of the list
  $\SSS_i$, so each of its $\ell\cdot p_i$ entries is repeated
  $a_{ij}$ times.  Similarly, let $\SSS^{(i,j)}_j$ be the
  concatenation of $b_{ij}$ copies of $\SSS_j$.
\end{notation}
The new configuration could be (but is not) be descriptively denoted
$(a_{ij}\SSS_i,b_{ij}\SSS_j)$.  So far, $\SSS^{(i,j)}$ is a
$((1,1),r,\ell\cdot\ell_{ij},D)$-configuration, since both
$\SSS^{(i,j)}_i$ and $\SSS^{(i,j)}_j$ have $\ell\cdot\ell_{ij}$
entries, and the independence property of each $r$-tuple is inherited.
\begin{lem}\label{lem6.1}
  If $(\SSS_i,\SSS_j)$ as above is a weight $(p_i,p_j)$
  h-configuration, then $\SSS^{(i,j)}$ is a weight $(1,1)$
  h-configuration.
\end{lem}
\begin{proof}
  Part 1.\ of Definition \ref{def4.2} is satisfied, with weight
  $(1,1)$: By construction, the $i$-degree of any point $z$ in the
  $\SSS^{(i,j)}$ configuration is $a_{ij}$ times $deg_i(z)$, the
  $i$-degree of the same point in the $\SSS$ configuration, and
  similarly for $j$, so:
  $$\frac{deg_i(z)\cdot a_{ij}}{1}=\frac{deg_j(z)\cdot
    b_{ij}}{1}\iff\frac{deg_i(z)}{p_i}=\frac{deg_j(z)}{p_j}.$$

  Dually, part 2.\ of Definition \ref{def4.2}
  is also satisfied, by the same calculation.
\end{proof}
The following identity applies Theorem \ref{thm4.4} to $\SSS^{(i,j)}$.
Recall $$h_{ij}:\fd P({\mathbf p})\to\fd
P^1:z\mapsto[z_i^{\ell_{ij}/p_i}:z_j^{\ell_{ij}/p_j}]$$ is the axis
projection (\ref{eq6}) from Lemma \ref{lem1.19}.
\begin{cor}\label{cor6.2}
  If $\SSS$ is a weight $\mathbf p$ h-configuration, then
  $$E_{(1,1)}(\SSS^{(i,j)})=h_{ij}\left(E_{\mathbf p}(\SSS)\right).$$
\end{cor}
\begin{proof}
  Suppose the points and projective $(r-1)$-subspaces in the weight
  $\mathbf p$ h-configuration $\SSS$ have been assigned vector
  representatives $\mathbf z$ and bases $\bbb_L$ as in Theorem
  \ref{thm4.4}.  In the weight $(1,1)$ h-configuration $\SSS^{(i,j)}$,
  we can use the same representatives and bases.  By the weight
  $(1,1)$ case of Theorem \ref{thm4.4}, using some choice of ordered
  representative $$S^{(i,j)}_i=\left(\vec s^{\
  (i,j)}_{i,1},\ldots,\vec s^{\ (i,j)}_{i,K},\ldots,\vec s^{\
  (i,j)}_{i,\ell\cdot\ell_{ij}}\right)$$ for $\SSS^{(i,j)}_i$ and
  similarly $S^{(i,j)}_j$ for $\SSS^{(i,j)}_j$, the following
  $E_{(1,1)}$ ratio in $\fd P^1$ is well-defined, and invariant under
  morphisms.  The products can be expanded using the multiplicity of
  the $r$-tuples.
\begin{eqnarray*}
  E_{(1,1)}(\SSS^{(i,j)})&=&\left[\prod_{K=1}^{\ell\cdot\ell_{ij}}\llbracket\vec
  s^{\
  (i,j)}_{i,K}\rrbracket_{\bbb_{i,K}}:\prod_{K=1}^{\ell\cdot\ell_{ij}}\llbracket\vec
  s^{\ (i,j)}_{j,K}\rrbracket_{\bbb_{j,K}}\right]\\
  &=&\left[\left(\prod_{K^\prime=1}^{\ell p_i}\llbracket\vec
  s_i^{K^\prime}\rrbracket_{\bbb_{i,K^\prime}}\right)^{a_{ij}}:\left(\prod_{K^\prime=1}^{\ell
  p_j}\llbracket\vec
  s_j^{K^\prime}\rrbracket_{\bbb_{j,K^\prime}}\right)^{b_{ij}}\right]\\
  &=&h_{ij}\left(E_{\mathbf p}(\SSS)\right).
\end{eqnarray*}
\end{proof}
The analogue of the above construction in classical Invariant Theory
is the formation of an absolute invariant as a ratio of powers of
differently weighted relative invariants, as in (\cite{s} Art.\
XII.122).

Suppose ${\mathbf p}$ and $\fd$ have the property that $\fd P({\mathbf
p})$ is reconstructible.  By Lemma \ref{lem1.22}, $E_{\mathbf
p}(\SSS)$ is uniquely determined by the set of ratios
$h_{ij}\left(E_{\mathbf p}(\SSS)\right)$, for $(i,j)\in I$.  Corollary
\ref{cor6.2} shows that the weight $\mathbf p$ invariant $E_{\mathbf
p}(\SSS)$ can be uniquely reconstructed by finding the weight $(1,1)$
invariant for all (or possibly fewer) of the weight $(1,1)$
h-configurations $\SSS^{(i,j)}$.  So, the $E_{\mathbf p}$ invariant
has no more power to distinguish projectively inequivalent weight
$\mathbf p$ h-configurations $\SSS$ than does the $E_{(1,1)}$
invariant, applied at most $n(n+1)/2$ times, two colors at a time, via
the above construction.

However, if $\fd P({\mathbf p})$ is not reconstructible, then there
may be weight $\mathbf p$ h-configurations with different $E_{\mathbf
p}$ invariants, but which cannot be distinguished using only
$E_{(1,1)}$ and the reconstruction process described in the previous
paragraph.  The following two Examples show this can happen when
$\fd=\re$, $r=2$, and Eves' Theorem is applied to signed distances in
$\re^D$ as in Corollary \ref{cor5.1}.
\begin{example}\label{ex6.1}
  The simplest example of a non-reconstructible weighted projective
  space is $\re P(2,2)$, where there is only one axis projection in
  the product from Definition \ref{def1.21}: let $h_{0,1}:\re
  P(2,2)\to\re P(1,1)$ be the two-to-one map induced by the inclusion
  ${\mathbf h}_{0,1}(z_0,z_1)=(z_0,z_1)$ as in Theorem \ref{thm1.12}
  and Example \ref{ex1.13}.  The simplest example of a weight $(2,2)$
  h-configuration has $r=2$, $D=1$ and $\ell=1$: one line $L=\re
  P^1$.  Let $\alpha$ and $\beta$ be distinct points on
  $\re^1\subseteq\re P^1$, and consider the configuration with the
  directed segment $(\alpha,\beta)$ appearing with multiplicity $4$:
  two black segments and two red segments.  The indexing as in
  (\ref{eq21}) is $\SSS=(\SSS_0,\SSS_1)$, and
  $\SSS_0=\SSS_1=[(\alpha,\beta),(\alpha,\beta)]$.  If we pick any
  unit of length in either direction, in order to define
  $\llbracket\overrightarrow{\alpha\beta}\rrbracket$ as the signed
  length of the directed segment $(\alpha,\beta)$, then the weighted
  invariant from Corollary \ref{cor5.1} is
  $$E_{(2,2)}(\SSS)=\left[\llbracket\overrightarrow{\alpha\beta}\rrbracket^2:\llbracket\overrightarrow{\alpha\beta}\rrbracket^2\right]_{(2,2)}=[1:1]_{(2,2)}.$$
  The modification of $\SSS$ into a weight $(1,1)$ h-configuration
  $\SSS^{(0,1)}$ is only a change in point of view from a
  $((2,2),2,1,1)$-configuration to a $((1,1),2,2,1)$-configuration;
  there is no change in the lists of segments:
  $$\SSS^{(0,1)}=(\SSS_0^{(0,1)},\SSS_1^{(0,1)})=(\SSS_0,\SSS_1)=\SSS,$$
  or the set of lines, $\{L\}$.  By Corollary \ref{cor6.2}, the $(1,1)$ invariant of this
  h-configuration is:
  $$E_{(1,1)}(\SSS^{(0,1)})=h_{0,1}(E_{(2,2)}(\SSS))=\left[\llbracket\overrightarrow{\alpha\beta}\rrbracket^2:\llbracket\overrightarrow{\alpha\beta}\rrbracket^2\right]_{(1,1)}=[1:1].$$

\begin{center}
  \begin{figure}
\begin{center}
    \includegraphics[scale=0.65]{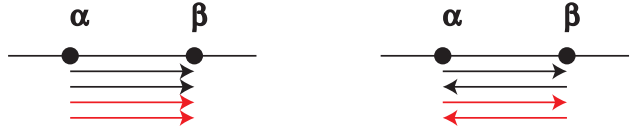}
    \caption{Configurations of $2$ points and $4$ segments on $1$ real
      line, from Example \ref{ex6.1}.  Left: configuration $\SSS$;
      Right: configuration $\ttt$.}\label{fig5}
  \end{center}
  \end{figure}
\end{center}

  Now, let $\ttt$ be a new weight $(2,2)$ h-configuration: the same
  line $L$ and points $\alpha$, $\beta$ as $\SSS$, but with two
  black segments in opposite directions, and two red segments also in
  opposite directions.  The indexing as in (\ref{eq21}) is
  $\ttt=(\ttt_0,\ttt_1)$,
  $\ttt_0=\ttt_1=[(\alpha,\beta),(\beta,\alpha)]$.  There is obviously
  no morphism $\SSS\to\ttt$, and the weighted invariant is a different
  element of $\re P(2,2)$:
  $$E_{(2,2)}(\ttt)=\left[\llbracket\overrightarrow{\alpha\beta}\rrbracket\llbracket\overrightarrow{\beta\alpha}\rrbracket:\llbracket\overrightarrow{\alpha\beta}\rrbracket\llbracket\overrightarrow{\beta\alpha}\rrbracket\right]_{(2,2)}=[-1:-1]_{(2,2)}.$$
  $\ttt$ is also a weight $(1,1)$ h-configuration, with $(1,1)$
  invariant:
  \begin{eqnarray}
    E_{(1,1)}(\ttt^{(0,1)})&=&h_{0,1}(E_{(2,2)}(\ttt))\nonumber\\
    &=&\left[\llbracket\overrightarrow{\alpha\beta}\rrbracket\llbracket\overrightarrow{\beta\alpha}\rrbracket:\llbracket\overrightarrow{\alpha\beta}\rrbracket\llbracket\overrightarrow{\beta\alpha}\rrbracket\right]_{(1,1)}\label{eq27}\\
    &=&[-1:-1]=[1:1].\nonumber
  \end{eqnarray}
  The conclusion is that the $E_{(1,1)}$ invariant cannot distinguish
  between $\SSS^{(0,1)}=\SSS$ and $\ttt^{(0,1)}=\ttt$.
\end{example}
\begin{example}\label{ex6.4}
  Let $\alpha$, $\beta$, $\gamma$ be the vertices of a triangle in the
  Euclidean plane $\re^2$, and let $\alpha^\prime$, $\beta^\prime$,
  $\gamma^\prime$ be the midpoints on opposite sides.  The following
  $((2,2,4),2,3,2)$-configuration $\SSS=(\SSS_0,\SSS_1,\SSS_2)$ is a
  weight $(2,2,4)$ h-configuration.
  \begin{eqnarray*}
    \SSS_0&=&[(\beta,\alpha^\prime),(\beta,\alpha^\prime),(\gamma,\beta^\prime),(\gamma,\beta^\prime),(\alpha,\gamma^\prime),(\alpha,\gamma^\prime)]\\
    \SSS_1&=&[(\alpha^\prime,\gamma),(\alpha^\prime,\gamma),(\beta^\prime,\alpha),(\beta^\prime,\alpha),(\gamma^\prime,\beta),(\gamma^\prime,\beta)]\\
    \SSS_2&=&[(\beta,\alpha^\prime),(\beta,\alpha^\prime),(\alpha^\prime,\gamma),(\alpha^\prime,\gamma),(\gamma,\beta^\prime),(\gamma,\beta^\prime),\\
    &&\ (\beta^\prime,\alpha),(\beta^\prime,\alpha),(\alpha,\gamma^\prime),(\alpha,\gamma^\prime),(\gamma^\prime,\beta),(\gamma^\prime,\beta)].
  \end{eqnarray*}
  It is possible to pick a direction and unit of length for each of
  the three lines so all the directed segments have signed length
  $+1$.  The invariant from Corollary \ref{cor5.1} is:
  \begin{eqnarray*}
    E_{\mathbf p}(\SSS)&=&[1:1:1]_{(2,2,4)}.
  \end{eqnarray*}

\begin{center}
  \begin{figure}
\begin{center}
    \includegraphics[scale=0.6]{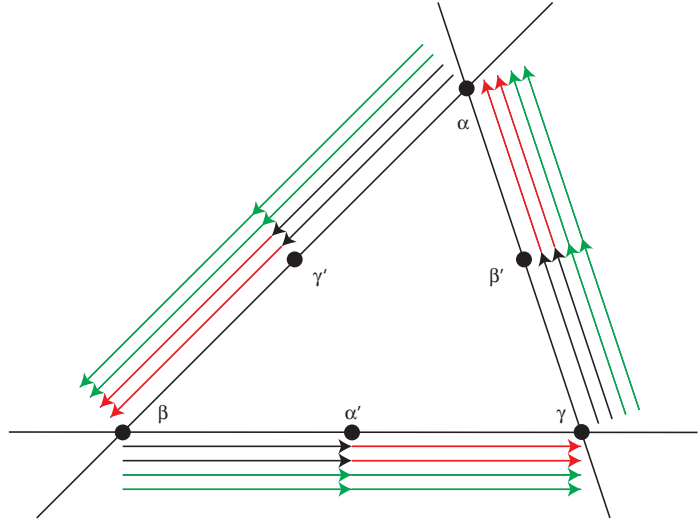}
    \caption{The configuration $\SSS$ from Example \ref{ex6.4}, of $6$
      points, $3$ lines, and $24$ directed segments in the real
      plane.}\label{fig7}
  \end{center}
  \end{figure}
\end{center}

  If we ignore the green segments and look only at the black and red
  segments, $\SSS^{(0,1)}=(\SSS_0,\SSS_1)$ is an h-configuration with
  $E_{(1,1)}(\SSS^{(0,1)})=[1:1]$.  However, the other color pairs
  $(\SSS_0,\SSS_2)$ and $(\SSS_1,\SSS_2)$ are not h-configurations.
  The modification of $(\SSS_0,\SSS_2)$ into $\SSS^{(0,2)}$ is to
  duplicate all the black segments, so each line has four black
  segments and four green segments.  Then
  $E_{(1,1)}(\SSS^{(0,2)})=[1:1]$ as in Corollary \ref{cor6.2}, and
  similarly $E_{(1,1)}(\SSS^{(1,2)})=[1:1]$.

  By Theorems \ref{thm1.12} and \ref{thm1.29}, the product of axis
  projections,
  \begin{eqnarray*}
    \prod h_{ij}:\re P(2,2,4)&\to&\re P^1\times\re P^1\times\re P^1:\\
    z&\mapsto&(h_{01}(z),h_{02}(z),h_{12}(z)),\\
    \mbox{$[z_0:z_1:z_2]_{(2,2,4)}$}&\mapsto&([z_1:z_2],[z_0^2:z_2],[z_1^2:z_2]),
  \end{eqnarray*}
  is two-to-one on $D_{\mathbf p}$.  In particular,
  $[1:1:1]_{(2,2,4)}\mapsto([1:1],[1:1],[1:1])$, and the other point
  with that image is $[-1:-1:1]_{(2,2,4)}$.

  So, as in Example \ref{ex6.1}, it is possible to find projectively
  inequivalent weight $(2,2,4)$ h-configurations $\SSS$ and $\ttt$
  with different $E_{(2,2,4)}$ invariants, but which have the same
  $E_{(1,1)}$ invariants from applying Eves' Theorem to their three
  h-configurations $\SSS^{(i,j)}$ and $\ttt^{(i,j)}$.  We can reverse
  some of the red and black directed segments from $\SSS$ to get a new
  configuration $\ttt=(\ttt_0,\ttt_1,\ttt_2)$,
  \begin{eqnarray*}
    \ttt_0&=&[(\beta,\alpha^\prime),(\alpha^\prime,\beta),(\gamma,\beta^\prime),(\beta^\prime,\gamma),(\alpha,\gamma^\prime),(\gamma^\prime,\alpha)]\\
    \ttt_1&=&[(\alpha^\prime,\gamma),(\gamma,\alpha^\prime),(\beta^\prime,\alpha),(\alpha,\beta^\prime),(\gamma^\prime,\beta),(\beta,\gamma^\prime)]\\
    \ttt_2&=&\SSS_2.
  \end{eqnarray*}

\begin{center}
  \begin{figure}
\begin{center}
    \includegraphics[scale=0.6]{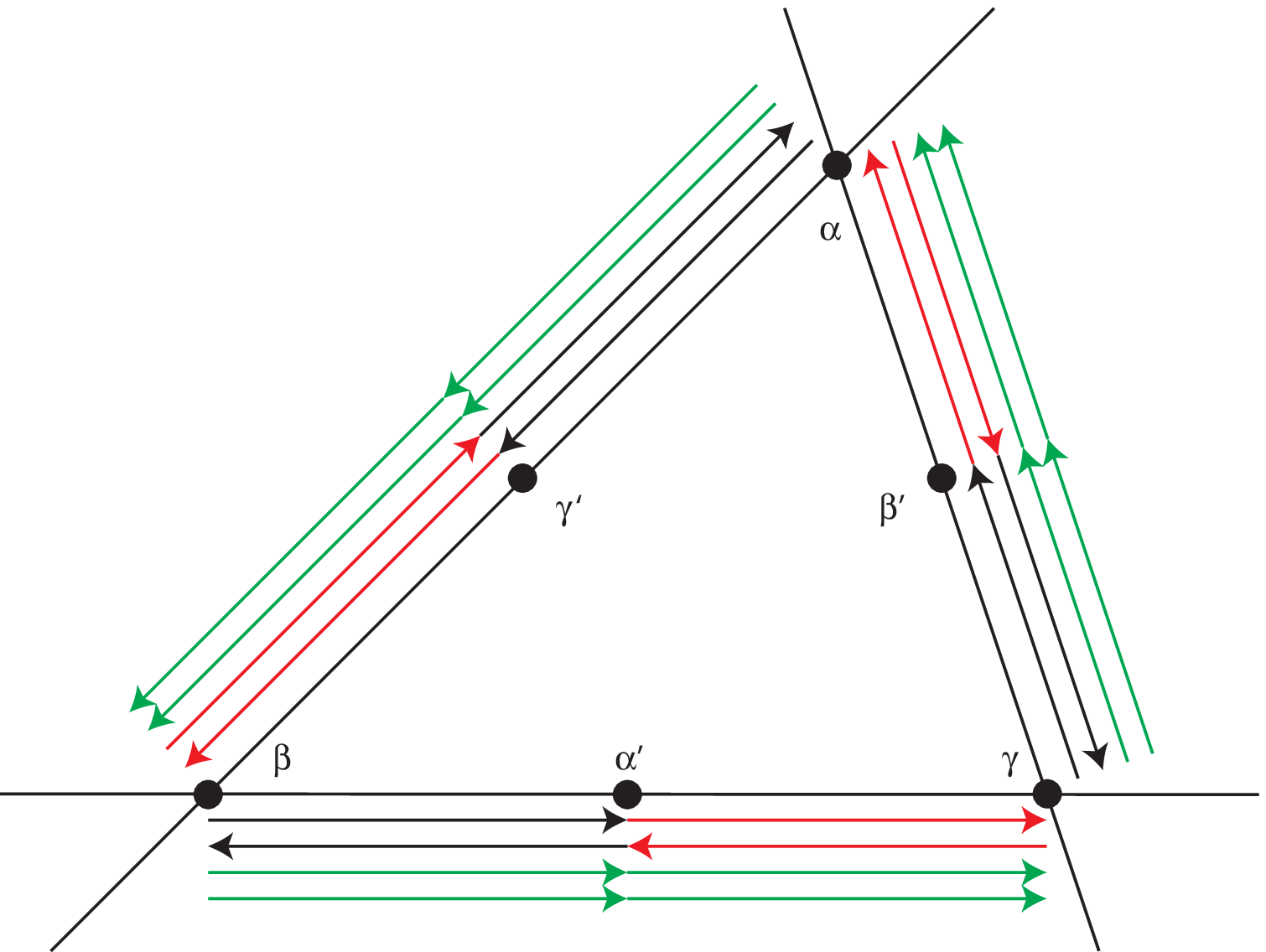}
    \caption{The configuration $\ttt$ from Example
      \ref{ex6.4}.}\label{fig8}
  \end{center}
  \end{figure}
\end{center}

  So, $E_{\mathbf p}(\ttt)=[-1:-1:1]_{(2,2,4)}$, and all three $(i,j)$
  color pairs have $E_{(1,1)}(\ttt^{(i,j)})=[1:1]$.
\end{example}
The next Example is a configuration considered by \cite{b}; all six
points are in the Euclidean plane, as in Examples \ref{ex5.2},
\ref{ex5.3}, but the configuration can be seen to have an octahedral
pattern.

\begin{center}
  \begin{figure}
\begin{center}
    \includegraphics[scale=0.6]{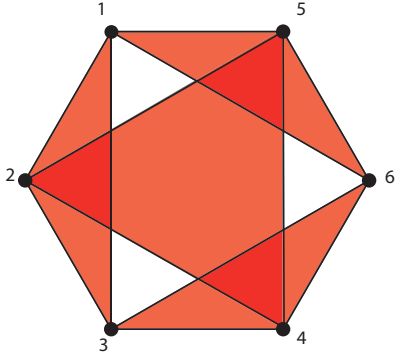}
    \caption{A configuration of $6$ points in the real plane, showing
      the $4$ red triangles from Example \ref{ex6.5}.}\label{fig6}
  \end{center}
  \end{figure}
\end{center}

\begin{example}\label{ex6.5}
  Let
  ${\mathnormal1},{\mathnormal2},{\mathnormal3},{\mathnormal4},{\mathnormal5},{\mathnormal6}$,
  be six points in the Euclidean plane as in Example \ref{ex5.2}.  Let
  $\SSS=(\SSS_0,\SSS_1)$ be a configuration of four black triangles
  and four red triangles:
  \begin{eqnarray*}
    \SSS_0&=&[\triangle({\mathnormal4}{\mathnormal6}{\mathnormal5}),\triangle({\mathnormal4}{\mathnormal2}{\mathnormal3}),\triangle({\mathnormal5}{\mathnormal1}{\mathnormal2}),\triangle({\mathnormal1}{\mathnormal3}{\mathnormal6})],\\
    \SSS_1&=&[\triangle({\mathnormal1}{\mathnormal2}{\mathnormal3}),\triangle({\mathnormal1}{\mathnormal6}{\mathnormal5}),\triangle({\mathnormal2}{\mathnormal4}{\mathnormal5}),\triangle({\mathnormal3}{\mathnormal4}{\mathnormal6})].
  \end{eqnarray*}
  Each of the six points has black degree and red degree equal to $2$,
  so, unlike Example \ref{ex5.2}, $\SSS$ can be viewed as either a
  $((2,2),3,2,2)$-configuration or a $((1,1),3,4,2)$-configuration.
  $\SSS$ is both a weight $(2,2)$ h-configuration and a weight $(1,1)$
  h-configuration, equal to $\SSS^{(0,1)}$.  In the plane coordinate
  system from Section \ref{sec4}, the $(2,2)$ invariant is:
  \begin{eqnarray*}
    E_{(2,2)}(\SSS)&=&\left[\llbracket\triangle({\mathnormal4}{\mathnormal6}{\mathnormal5})\rrbracket_\bbb\llbracket\triangle({\mathnormal4}{\mathnormal2}{\mathnormal3})\rrbracket_\bbb\llbracket\triangle({\mathnormal5}{\mathnormal1}{\mathnormal2})\rrbracket_\bbb\llbracket\triangle({\mathnormal1}{\mathnormal3}{\mathnormal6})\rrbracket_\bbb:\right.\\
     && \
     \left.\llbracket\triangle({\mathnormal1}{\mathnormal2}{\mathnormal3})\rrbracket_\bbb\llbracket\triangle({\mathnormal1}{\mathnormal6}{\mathnormal5})\rrbracket_\bbb\llbracket\triangle({\mathnormal2}{\mathnormal4}{\mathnormal5})\rrbracket_\bbb\llbracket\triangle({\mathnormal3}{\mathnormal4}{\mathnormal6})\rrbracket_\bbb\right]_{(2,2)}\\
    &=&[z_0:z_1]_{(2,2)}.
  \end{eqnarray*}
  The $(1,1)$ invariant of the same configuration is:
  $$E_{(1,1)}(\SSS^{(0,1)})=[z_0:z_1]_{(1,1)},$$ which can be
  interpreted as the ratio of signed areas:
  $$\frac{z_1}{z_0}=\frac{(Area\triangle({\mathnormal1}{\mathnormal2}{\mathnormal3}))(Area\triangle({\mathnormal1}{\mathnormal6}{\mathnormal5}))(Area\triangle({\mathnormal2}{\mathnormal4}{\mathnormal5}))(Area\triangle({\mathnormal3}{\mathnormal4}{\mathnormal6}))}{(Area\triangle({\mathnormal4}{\mathnormal6}{\mathnormal5}))(Area\triangle({\mathnormal4}{\mathnormal2}{\mathnormal3}))(Area\triangle({\mathnormal5}{\mathnormal1}{\mathnormal2}))(Area\triangle({\mathnormal1}{\mathnormal3}{\mathnormal6}))}.$$
  We remark that the property $E_{(1,1)}(\SSS^{(0,1)})=[1:1]$ is
  equivalent to the projective property that six points lie on a conic
  (\cite{crg}, \cite{rg}).

  Let $\ttt=(\ttt_0,\ttt_1)$ be a new configuration --- the same six
  points, but changing the order of points in one of the black
  triangles and one of the red triangles to get the opposite signed
  areas:
  \begin{eqnarray*}
    \ttt_0&=&[\triangle({\mathnormal4}{\mathnormal5}{\mathnormal6}),\triangle({\mathnormal4}{\mathnormal2}{\mathnormal3}),\triangle({\mathnormal5}{\mathnormal1}{\mathnormal2}),\triangle({\mathnormal1}{\mathnormal3}{\mathnormal6})],\\
    \ttt_1&=&[\triangle({\mathnormal1}{\mathnormal3}{\mathnormal2}),\triangle({\mathnormal1}{\mathnormal6}{\mathnormal5}),\triangle({\mathnormal2}{\mathnormal4}{\mathnormal5}),\triangle({\mathnormal3}{\mathnormal4}{\mathnormal6})].
  \end{eqnarray*}
  Then $\ttt$ has the same $E_{(1,1)}$ invariant, and in the ratio of
  signed areas, the sign changes cancel, giving the same ratio as
  $\SSS$.  The two configurations have different $(2,2)$ invariants,
  so they are projectively inequivalent:
  $$E_{(2,2)}(\ttt)=[-z_0:-z_1]_{(2,2)}\ne[z_0:z_1]_{(2,2)}=E_{(2,2)}(\SSS).$$
\end{example}

\subsubsection*{Acknowledgments}
  The author was motivated to start writing this paper after attending
  a talk by M.\ Frantz (\cite{f2}), and he thanks Frantz for
  subsequent conversations and for pointing out the reference
  \cite{p}.  The Figures were generated using \cite{AI}.

\end{document}